\documentclass[11pt]{amsart} \textwidth=14.5cm \oddsidemargin=1cm
\usepackage{amsmath}
\usepackage{amsxtra}
\usepackage{amscd}
\usepackage{amsthm}
\usepackage{amsfonts}
\usepackage{amssymb}
\usepackage{eucal}

\input prepictex
\input pictex
\input postpictex

\newtheorem{thm}{Theorem}[section]
\newtheorem{cor}[thm]{Corollary}
\newtheorem{lem}[thm]{Lemma}
\newtheorem{prop}[thm]{Proposition}

\theoremstyle{definition}

\newtheorem{example}[thm]{Example}

\theoremstyle{remark}
\newtheorem{rem}[thm]{Remark}

\numberwithin{equation}{section}

\begin{document}

\newcommand{\thmref}[1]{Theorem~\ref{#1}}
\newcommand{\secref}[1]{Section~\ref{#1}}
\newcommand{\lemref}[1]{Lemma~\ref{#1}}
\newcommand{\propref}[1]{Proposition~\ref{#1}}
\newcommand{\corref}[1]{Corollary~\ref{#1}}
\newcommand{\remref}[1]{Remark~\ref{#1}}
\newcommand{\eqnref}[1]{(\ref{#1})}
\newcommand{\exref}[1]{Example~\ref{#1}}

\newcommand{\nc}{\newcommand}
\nc{\Z}{{\mathbb Z}}
\nc{\C}{{\mathbb C}}
\nc{\N}{{\mathbb N}}
\nc{\F}{{\mf F}}
\nc{\Q}{\ol{Q}}
\nc{\la}{\lambda}
\nc{\ep}{\epsilon}
\nc{\h}{\mathfrak h}
\nc{\n}{\mf n}
\nc{\A}{{\mf a}}
\nc{\G}{{\mathfrak g}}
\nc{\SG}{\overline{\mathfrak g}}
\nc{\DG}{\widetilde{\mathfrak g}}
\nc{\D}{\mc D} \nc{\Li}{{\mc L}} \nc{\La}{\Lambda} \nc{\is}{{\mathbf
i}} \nc{\V}{\mf V} \nc{\bi}{\bibitem} \nc{\NS}{\mf N}
\nc{\dt}{\mathord{\hbox{${\frac{d}{d t}}$}}} \nc{\E}{\mc E}
\nc{\ba}{\tilde{\pa}} \nc{\half}{\frac{1}{2}} \nc{\mc}{\mathcal}
\nc{\mf}{\mathfrak} \nc{\hf}{\frac{1}{2}}
\nc{\hgl}{\widehat{\mathfrak{gl}}} \nc{\gl}{{\mathfrak{gl}}}
\nc{\hz}{\hf+\Z} \nc{\vac}{|0 \rangle}
\nc{\dinfty}{{\infty\vert\infty}} \nc{\SLa}{\overline{\Lambda}}
\nc{\SF}{\overline{\mathfrak F}} \nc{\SP}{\overline{\mathcal P}}
\nc{\U}{\mathfrak u} \nc{\SU}{\overline{\mathfrak u}}
\nc{\ov}{\overline}
\nc{\wt}{\widetilde}
\nc{\sL}{\ov{\mf{l}}}
\nc{\sP}{\ov{\mf{p}}}
\nc{\osp}{\mf{osp}}
\nc{\sdeg}{\ov{\rm deg}}
\nc{\spo}{\mf{spo}}
\nc{\hosp}{\widehat{\mf{osp}}}
\nc{\hspo}{\widehat{\mf{spo}}}

\advance\headheight by 2pt

\title[Irreducible Characters of Superalgebra and Super Duality]{Irreducible Characters of General Linear Superalgebra and Super Duality}

\author[Shun-Jen Cheng]{Shun-Jen Cheng$^\dagger$}
\thanks{$^\dagger$Partially supported by an NSC-grant and an Academia Sinica Investigator grant}
\address{Institute of Mathematics, Academia Sinica, Taipei,
Taiwan 10617} \email{chengsj@math.sinica.edu.tw}

\author[Ngau Lam]{Ngau Lam$^{\dagger\dagger}$}
\thanks{$^{\dagger\dagger}$Partially supported by an NSC-grant}
\address{Department of Mathematics, National Cheng-Kung University, Tainan, Taiwan 70101}
\email{nlam@mail.ncku.edu.tw}

\begin{abstract} \vspace{.3cm}
We develop a new method to solve the irreducible character problem for a wide class
of modules over the general linear superalgebra, including all the finite-dimensional modules, by directly
relating the problem to the classical Kazhdan-Lusztig theory. Furthermore, we prove that certain parabolic BGG categories over the general linear algebra and over the general linear superalgebra are equivalent.   We also verify a parabolic version of a
conjecture of Brundan on the irreducible characters in the BGG category of the
general linear superalgebra.
\end{abstract}



 \maketitle


\section{Introduction}

The problem of finding the finite-dimensional irreducible characters of simple Lie
superalgebras was first posed in \cite{K1, K2}.  This problem turned out to be one of the
most challenging problems in the theory of Lie superalgebras, and in the type $A$ case
was first solved by Serganova \cite{Se}.  Later on, inspired by \cite{LLT}, Brundan in
\cite{B} provided an elegant new solution of the problem. To be more precise, Brundan in
\cite[Conjecture 4.32 and (4.35)]{B} gave a conjectural character formula for every
irreducible highest weight $\gl(m|n)$-module in the Bernstein-Gelfand-Gelfand category
$\ov{\mc{O}}$ in terms of certain Brundan-Kazhdan-Lusztig polynomials. The validity of
the conjecture would imply a remarkable formulation of the Kazhdan-Lusztig theory of
$\ov{\mc O}$ in terms of canonical and dual basis on a certain Fock space. Brundan then
solved the finite-dimensional irreducible character problem by verifying the conjecture
for the subcategory of finite-dimensional $\gl(m|n)$-modules, in which case the Fock
space is $\widehat{\mc{E}}^{m|n}$ (see \secref{KL:polynomials}). One of the main purposes
of the present paper is to establish Brundan's conjecture for a substantially larger
subcategory of $\ov{\mc O}$ of $\gl(m|n)$-modules, which includes all the
finite-dimensional ones. We note that a similar Fock space formulation is known among
experts for modules of the general linear algebra $\gl(m+n)$ in the category $\mc{O}$,
and in particular for modules in the maximal parabolic subcategory corresponding to the
Levi subalgebra $\gl(m)\oplus\gl(n)$, in which case the Fock space is
$\widehat{\mc{E}}^{m+n}$ (see \secref{KL:polynomials}).

Let $\G$ and $\SG$ denote direct limits of general linear algebras $\gl(m+n)$ and of
general linear superalgebras $\gl(m|n)$, respectively, as $n\to\infty$ (see
\secref{aux222} and \secref{aux333}). Motivated by \cite{B} it was shown in \cite{CWZ}
that in the limit $n\to\infty$ the Fock spaces $\widehat{\mc{E}}^{m+n}$ and
$\widehat{\mc{E}}^{m|n}$ have compatible canonical and dual canonical bases, and that the
Kazhdan-Lusztig polynomials in $\widehat{\mc{E}}^{m+\infty}$ and
$\widehat{\mc{E}}^{m|\infty}$ can be identified. Now the Kazhdan-Lusztig polynomials of
$\widehat{\mc{E}}^{m+\infty}$ describe the $\G$-module category $\mc{O}^f_{[-m,-2]}$
whose objects are the direct limits of modules in the above-mentioned maximal parabolic
subcategory, while those of $\widehat{\mc{E}}^{m|\infty}$ describe the $\SG$-module
category $\ov{\mc{O}}^f_{[-m,-2]}$ whose objects are direct limits of finite-dimensional
$\gl(m|n)$-modules. From this and Brundan's formulation it follows that the classical
(parabolic) Kazhdan-Lusztig polynomials of the general linear algebra also give solution
to the finite-dimensional irreducible character problem for the general linear
superalgebra.

Motivated by \cite{CWZ} Wang and the first author in \cite{CW2} compare a more general
parabolic $\G$-module category ${\mc{O}}^f_Y$ with a corresponding $\SG$-module category
$\ov{\mc{O}}^f_Y$, where $Y$ here is any subset of $[-m,-2]$ (see Sections \ref{aux222},
\ref{aux333}, and \remref{rmk:intro}). A precise statement of the parabolic Brundan
conjecture (\cite[Conjecture 4.32]{B}) on the character of irreducible $\SG$-modules in
$\ov{\mc{O}}^f_Y$ was given in \cite[Conjecture 3.10]{CW2}. The results in \cite{CWZ,
CW2} suggest a direct connection between the categories $\mc{O}^f_Y$ and
$\ov{\mc{O}}^f_Y$. In fact, the categories $\mc{O}^f_Y$ and $\ov{\mc{O}}^f_Y$ are
conjectured to be equivalent in \cite[Conjecture 4.18]{CW2}, which was referred to as
{\em super duality}.

The purpose of the present paper is to establish this super duality. Our main idea is the
introduction of a bigger Lie superalgebra $\DG$ (\secref{aux122}), which contains and
interpolates $\G$ and $\SG$.  We then study a corresponding category $\wt{\mc{O}}^f_Y$ of
$\DG$-modules and define certain truncation functors
$T:\wt{\mc{O}}^f_Y\rightarrow\mc{O}^f_Y$ and
$\ov{T}:\wt{\mc{O}}^f_Y\rightarrow\ov{\mc{O}}^f_Y$ (\secref{Tfunctors}).  These functors
are shown to send parabolic Verma $\DG$-modules to the respective parabolic Verma $\G$-
and $\SG$-modules, and furthermore irreducible $\DG$-modules to the respective
irreducible $\G$- and $\SG$-modules. From this we obtain in \thmref{character} a solution
of the irreducible character problem for $\SG$-modules in $\ov{\mc{O}}^f_Y$.  The
solution of the irreducible character problem then allows us to compare the
Kazhdan-Lusztig polynomials in $\ov{\mc{O}}^f_Y$ with those in $\mc{O}^f_Y$.  This then
enables us to prove in \thmref{thm:equivalence} that the functors $T$ and $\ov{T}$ define
equivalences of categories from which super duality follows.

Note that a special case of the super duality conjecture was already formulated for the
categories $\mc{O}^f_{[-m,-2]}$ and $\ov{\mc{O}}^f_{[-m,-2]}$ in \cite[Conjecture
6.10]{CWZ}. A proof of this special case was announced recently in \cite{BS}, with a
proof to appear in a sequel of \cite{BS}.  Our method differs significantly from that of
Brundan and Stroppel, as ours is independent of \cite{B}. Furthermore our approach
enables us to explicitly construct functors inducing this equivalence, and it is
applicable to more general module categories.

We want to emphasize that, in contrast to \cite{CWZ}, the arguments presented in this
article do not depend on \cite{B} or \cite {Se}, and hence \thmref{character} also gives
an independent solution to the finite-dimensional irreducible character problem for the
general linear superalgebra as a special case. By directly relating the irreducible
character problem of Lie superalgebras to that of Lie algebras our solution of the
problem becomes surprisingly elementary. Our method is applicable to other finite and
infinite-dimensional superalgebras, e.g. the ortho-symplectic Lie superalgebras
\cite{CLW}.

This article is organized as follows. In \secref{section2} the Lie
superalgebras $\DG$, $\G$ and $\SG$ are defined, together with the
module categories $\wt{\mc{O}}^f_Y$, $\mc{O}^f_Y$ and
$\ov{\mc{O}}_Y^f$.  In \secref{section3} the main tool, odd
reflections \cite{LSS}, of making connections between these
categories is introduced, and the crucial \lemref{lem:change}
is proved. In \secref{section4} we show that the Kazhdan-Lusztig
polynomials of these categories coincide, from which we then
derive in \secref{section5} the equivalence of these categories.

We conclude this introduction by setting the notation to be used throughout this article.
The symbols $\Z$, $\N$, and $\Z_+$ stand for the sets of all, positive and non-negative
integers, respectively. For $m\in\Z$ we set $\langle m\rangle:=m$, if $m> 0$, and
$\langle m\rangle:=0$, otherwise. For integers $a<b$ we set $[a,b]:=\{a,a+1,\cdots,b\}$.
Let $\mc{P}$ denote the set of partitions.  For $\la\in\mc{P}$ we denote by $\la'$ the
transpose partition of $\la$, by $\ell(\la)$ the length of $\la$ and by
$s_\la(y_1,y_2,\cdots)$ the Schur function in the indeterminates $y_1,y_2,\cdots$
associated with $\la$. For a super space $V=V_{\bar{0}}\oplus V_{\bar{1}}$ and a
homogeneous element $v\in V$, we use the notation $|v|$ to denote the $\Z_2$-degree of
$v$. Let $\mc{U}(\G)$ denote the universal enveloping algebra of a Lie (super)algebra
$\G$. Finally all vector spaces, algebras, tensor products, et cetera, are over the field
of complex numbers $\C$.

\smallskip

\noindent{\bf Acknowledgments.} We are very grateful to Weiqiang Wang for numerous
helpful comments and suggestions. We also thank one of the referees for suggestions.

\section{The Lie Superalgebras $\G$, $\SG$ and $\widetilde{\G}$}\label{section2}

\subsection{The Lie superalgebra $\DG$}\label{aux122}

For $m\in \N$, let $\wt{V}$ denote the complex super
space with homogeneous basis $\{v_{r}\vert r\in[-m,-1]\cup\hf\N\}$. The
$\Z_2$-gradation is determined by $|v_r|=\bar{1}$, for
$r\in\hf+\Z_+$, and $|v_i|=\bar{0}$, for $i\in[-m,-1]\cup\N$.  We
denote by $\widetilde{\G}$ the Lie
superalgebra of endomorphisms of $\wt{V}$ vanishing
on all but finitely many $v_r$s. For $r, s,p\in[-m,-1]\cup\hf\N$,
let $E_{rs}$ denote the endomorphism defined by
$E_{rs}(v_p):=\delta_{sp}v_r$. Then $\widetilde{\G}$ equals the
Lie superalgebra spanned by these $E_{rs}$s. Let $\DG^{<0}$ be
the subalgebra isomorphic to the linear algebra $\gl(m)$ spanned by $E_{ij}$,
$i,j\in[-m,-1]$.

Let $\widetilde{\h}$ stand for the Cartan subalgebra spanned by
the $E_{rr}$s and let $\{\epsilon_r\in\widetilde{\h}^*\vert
r\in[-m,-1]\cup\hf\N\}$ be the basis dual to
$\{E_{rr}\in\widetilde{\h}\vert r\in [-m,-1]\cup\hf\N\}$.  Let
$\widetilde{\Pi}$ denote the simple roots
$\{\alpha_{-m}:=\epsilon_{-m}-\epsilon_{-m+1},\cdots,\alpha_{-1}:=\epsilon_{-1}-\epsilon_{1/2}\}\cup
\{\alpha_r:=\epsilon_r-\epsilon_{r+\hf}\vert r\in\hf\N\}$. The
corresponding Dynkin diagram is

\begin{center}
\hskip -3cm \setlength{\unitlength}{0.16in}
\begin{picture}(24,3)
\put(6,2){\makebox(0,0)[c]{$\bigcirc$}}
\put(8.4,2){\makebox(0,0)[c]{$\bigcirc$}}
\put(12.85,2){\makebox(0,0)[c]{$\bigcirc$}}
\put(15.25,2){\makebox(0,0)[c]{$\bigotimes$}}
\put(17.4,2){\makebox(0,0)[c]{$\bigotimes$}}
\put(21.9,2){\makebox(0,0)[c]{$\bigotimes$}}
\put(6.4,2){\line(1,0){1.55}} \put(8.82,2){\line(1,0){0.8}}
\put(11.2,2){\line(1,0){1.2}} \put(13.28,2){\line(1,0){1.45}}
\put(15.7,2){\line(1,0){1.25}} \put(17.8,2){\line(1,0){0.9}}
\put(20.1,2){\line(1,0){1.4}}
\put(22.4,2){\line(1,0){1.4}}
\put(10.5,1.95){\makebox(0,0)[c]{$\cdots$}} \put(19.5,1.95){\makebox(0,0)[c]{$\cdots$}}
\put(24.5,1.95){\makebox(0,0)[c]{$\cdots$}}
\put(6.3,1){\makebox(0,0)[c]{\tiny $\alpha_{-m}$}} \put(8.9,1){\makebox(0,0)[c]{\tiny
$\alpha_{-m+1}$}} \put(12.8,1){\makebox(0,0)[c]{\tiny $\alpha_{-2}$}}
\put(15.15,1){\makebox(0,0)[c]{\tiny $\alpha_{-1}$}} \put(17.8,1){\makebox(0,0)[c]{\tiny
$\alpha_{1/2}$}} \put(22,1){\makebox(0,0)[c]{\tiny $\alpha_{n}$}}
\put(-1.5,2){\makebox(0,0)[c]{(D1)}}
\end{picture}
\end{center}
For $\alpha\in \widetilde{\Pi}$ let $\alpha^\vee$ denote the simple coroot corresponding
to $\alpha$. Explicitly, we have $\alpha^\vee_{-1}=E_{-1,-1}+E_{1/2,1/2}$,
$\alpha^\vee_{1/2}=-E_{1/2,1/2}-E_{11}$, $\alpha^\vee_{1}=E_{11}+E_{3/2,3/2}$,
$\alpha^\vee_{3/2}=-E_{3/2,3/2}-E_{22}$ et cetera.

Given $\alpha\in \widetilde{\h}^*$, let
$\widetilde{\G}_\alpha:=\{x\in \widetilde{\G}\mid\,
[h,x]=\alpha(h)x, \forall h\in \widetilde{\h}\}$ and let
$\widetilde{\Delta}$ denote the set of all roots. The positive
roots with respect to $\widetilde{\Pi}$ will be denoted by
$\widetilde{\Delta}_+$, while ${\wt{\mf{b}}}$ denotes the Borel
subalgebra with respect to $\wt{\Delta}_+$. Let
$\wt{\mf{n}}:=[{\wt{\mf{b}}},{\wt{\mf{b}}}]$ and let
$\wt{\mf{n}}_-$ be the opposite nilradical.

For any subset $Y\subseteq[-m,-2]$ (including $Y=\emptyset$),
define
 \begin{align*}
 \wt{\mf{l}}_Y&:=\widetilde{\h}\oplus(\oplus_{\alpha\in
 \widetilde{\Delta}_Y}\widetilde{\G}_\alpha),\\
 \wt{\mf{u}}_Y&:=\oplus_{\alpha\in
 \widetilde{\Delta}_+\backslash(\widetilde{\Delta}_Y)_+}\widetilde{\G}_\alpha,\\
 (\wt{\mf{u}}_Y)_-&:=\oplus_{\alpha\in
 \widetilde{\Delta}_+\backslash(\widetilde{\Delta}_Y)_+}\widetilde{\G}_{-\alpha},\\
 \wt{\mf{p}}_Y&:=\wt{\mf{l}}_Y\oplus\wt{\mf{u}}_Y,
 \end{align*}
where $\widetilde{\Delta}_Y:=\widetilde{\Delta}\cap(\oplus_{r\in
Y\cup\hf\N}\Z\alpha_r)$ and
$(\widetilde{\Delta}_Y)_+:=\widetilde{\Delta}_+\cap\widetilde{\Delta}_Y$.
Then $\wt{\mf{p}}_Y$ is a parabolic subalgebra of $\wt{\G}$ with
Levi subalgebra $\wt{\mf{l}}_Y$ and nilpotent radical
$\wt{\mf{u}}_Y$. Set
$\wt{\mf{p}}_Y^{<0}:=\DG^{<0}\cap\wt{\mf{p}}_Y$ and
$\wt{\mf{l}}_Y^{<0}:=\DG^{<0}\cap\wt{\mf{l}}_Y$.

Given $\la\in\widetilde{\h}^*$, we denote by
$L(\wt{\mf{l}}_Y,\la)$ the irreducible $\wt{\mf{l}}_Y$-module of
highest weight $\la$ with respect to
$\wt{\mf{l}}_Y\cap\wt{\mf{b}}$, which we may regard as an
irreducible $\wt{\mf{p}}_Y$-module in the usual way. Define the
{\em parabolic Verma $\DG$-module}
\begin{equation*}
\wt{K}(\la):={\rm
Ind}_{\wt{\mf{p}}_Y}^{\wt{\G}}L(\wt{\mf{l}}_Y,\la).
\end{equation*} Let $\wt{L}(\la)$ be the irreducible $\widetilde{\G}$-module of
highest weight $\la$ with respect to $\wt{\mf{b}}$.

Set
\begin{align*}\mc{P}_Y:=&\{\la=(\la_{-m},\la_{-m+1},\cdots,
\la_{-1},\la_1,\la_2,\cdots)\, \mid\\
&\la_i\in\Z,\,\forall i;\,
\langle\sum_{i=-m}^{-1}\la_i\epsilon_i,\alpha^\vee_j\rangle\in\Z_+,
\forall j\in Y;\,(\la_1,\la_2,\cdots)\in \mc{P}\}.\end{align*}
 For $\la\in\mc{P}_Y$ we let $\la_+:=(\la_1,\la_2,\cdots)\in \mc{P}$
and $\la^{<0}:=\sum_{i=-m}^{-1}\la_i\epsilon_i$.

 Given a partition
$\mu=(\mu_1,\mu_2,\cdots)$ we set $\theta(\mu)$ to be the sequence
of integers
\begin{equation*}
\theta(\mu):=(\theta(\mu)_{1/2},\theta(\mu)_1,\theta(\mu)_{3/2},\theta(\mu)_2,\cdots),
\end{equation*}
where $\theta(\mu)_{i-1/2}:=\langle\mu'_i-(i-1)\rangle$ and
$\theta(\mu)_i:=\langle\mu_i-i\rangle$, for $i\in\N$. We recall that $\langle
\cdot\rangle$ is defined at the end of the Introduction.

\begin{example}\label{eg:theta}
For the partition $\la =(7,6,3,3,1)$, we have $\theta(\la)=(5,6,3,4,2,0,0,\cdots)$. This
can be read off the Young diagram of $\la$ as follows:
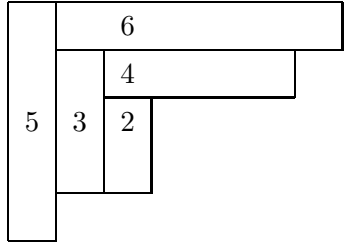
\begin{center}
\hskip 0cm \setlength{\unitlength}{0.25in}
\begin{picture}(7,5.5)
\put(0,0){\line(0,1){5}}
\put(1,0){\line(0,1){5}}
\put(2,1){\line(0,1){3}}
\put(3,1){\line(0,1){2}}
\put(6,3){\line(0,1){1}}
\put(7,4){\line(0,1){1}}
\put(0,0){\line(1,0){1}} \put(1,1){\line(1,0){2}} \put(3,3){\line(1,0){2}}
\put(2,3){\line(1,0){4}} \put(1,4){\line(1,0){6}} \put(0,5){\line(1,0){7}}
\put(2.5,4.5){\makebox(0,0)[c]{$6$}} \put(2.5,3.5){\makebox(0,0)[c]{$4$}}
\put(2.5,2.5){\makebox(0,0)[c]{$2$}} \put(0.5,2.5){\makebox(0,0)[c]{$5$}}
\put(1.5,2.5){\makebox(0,0)[c]{$3$}}
\end{picture}
\end{center}
\end{example}
Set
\begin{align*}
&\wt{\mc{P}}_Y:=\{\la^{<0}+\sum_{r\in\hf\N}\theta(\la_+)_r\epsilon_r\in{\wt{\h}}^*\,
\mid \,\la=(\la_i)\in\mc{P}_Y\}.
\end{align*}

For a semisimple $\wt{\h}$-module $\wt{M}$ and $\gamma\in
\wt{\h}^*$, we define
$$\wt{M}_\gamma:=\{m\in\wt{M}\vert
hm=\gamma(h)m,\forall h\in\wt{\h}\}.$$

\subsection{The subalgebra $\G$}\label{aux222}

We denote by $\G$ the subalgebra of $\wt{\G}$
generated by $E_{rs}$, $r,s\in[-m,-1]\cup\N$.  The Cartan
subalgebra $\h$ has basis $\{E_{ii}\vert i\in[-m,-1]\cup\N\}$, with
dual basis $\{\epsilon_i\vert i\in[-m,-1]\cup\N\}$. The
corresponding Dynkin diagram is
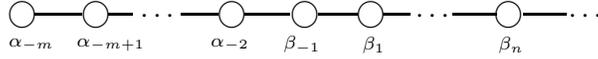
\begin{center}
\hskip -3cm \setlength{\unitlength}{0.16in}
\begin{picture}(24,3)
\put(6,2){\makebox(0,0)[c]{$\bigcirc$}}
\put(8.4,2){\makebox(0,0)[c]{$\bigcirc$}}
\put(12.85,2){\makebox(0,0)[c]{$\bigcirc$}}
\put(15.25,2){\makebox(0,0)[c]{$\bigcirc$}}
\put(17.4,2){\makebox(0,0)[c]{$\bigcirc$}}
\put(21.9,2){\makebox(0,0)[c]{$\bigcirc$}}
\put(6.4,2){\line(1,0){1.55}} \put(8.82,2){\line(1,0){0.8}}
\put(11.2,2){\line(1,0){1.2}} \put(13.28,2){\line(1,0){1.45}}
\put(15.7,2){\line(1,0){1.25}} \put(17.8,2){\line(1,0){0.9}}
\put(20.1,2){\line(1,0){1.4}}
\put(22.4,2){\line(1,0){1.4}}
\put(10.5,1.95){\makebox(0,0)[c]{$\cdots$}}
\put(19.5,1.95){\makebox(0,0)[c]{$\cdots$}}
\put(24.5,1.95){\makebox(0,0)[c]{$\cdots$}}
\put(6.3,1){\makebox(0,0)[c]{\tiny $\alpha_{-m}$}}
\put(8.9,1){\makebox(0,0)[c]{\tiny $\alpha_{-m+1}$}}
\put(12.8,1){\makebox(0,0)[c]{\tiny $\alpha_{-2}$}}
\put(15.15,1){\makebox(0,0)[c]{\tiny $\beta_{-1}$}}
\put(17.5,1){\makebox(0,0)[c]{\tiny $\beta_{1}$}}
\put(22,1){\makebox(0,0)[c]{\tiny $\beta_{n}$}}
\end{picture}
\end{center}
where
\begin{equation*}
\beta_{-1}:=\epsilon_{-1}-\epsilon_1,\quad
\beta_i:=\epsilon_i-\epsilon_{i+1},\quad i\ge 1.
\end{equation*}
Set $\mf{b}:={\wt{\mf{b}}}\cap\G$, ${\mf{n}}:={\wt{\mf{n}}}\cap\G$ and
${\mf{n}}_-:={\wt{\mf{n}}}_-\cap\G$. Set
$\mf{l}_Y:=\wt{\mf{l}}_Y\cap\G$, $\mf{p}_Y:=\wt{\mf{p}}_Y\cap\G$,
$\mf{u}_Y:=\wt{\mf{u}}_Y\cap\G$, and
$(\mf{u}_Y)_-:=(\wt{\mf{u}}_Y)_-\cap\G$.

Given $\la\in{\h}^*$, we denote by $L({\mf{l}}_Y,\la)$ the
irreducible ${\mf{l}}_Y$-module of highest weight $\la$ with
respect to $\mf{l}_Y\cap\mf{b}$, which we may regard as an
irreducible ${\mf{p}}_Y$-module in the usual way. Define the
parabolic Verma $\G$-module
\begin{equation*}
{K}(\la):={\rm Ind}_{\mf{p}_Y}^{{\G}}L({\mf{l}}_Y,\la).
\end{equation*} Let ${L}(\la)$ be the irreducible ${\G}$-module of
highest weight $\la$ with respect to $\mf{b}$.

As usual, we identify $\mc{P}_Y$ with the following set of weights in $\h^*$:
\begin{equation*}
\mc{P}_Y=\{\la^{<0}+\sum_{j\in\N}\la_j\epsilon_j\in\h^*\mid\,\la=(\la_i)\in\mc{P}_Y\}.
\end{equation*}


For a semisimple ${\h}$-module ${M}$ and $\gamma\in {\h}^*$, we
define
$${M}_\gamma:=\{m\in {M}\vert hm=\gamma(h)m,\forall h\in{\h}\}.$$

\subsection{The subalgebra $\SG$}\label{aux333}

We denote by $\SG$ the subalgebra of $\wt{\G}$
generated by $E_{rs}$, $r,s\in[-m,-1]\cup\hf+\Z_+$. The Cartan
subalgebra $\ov{\h}$ has basis $\{E_{rr}\vert
r\in[-m,-1]\cup\hf+\Z_+\}$, with dual basis $\{\epsilon_r\vert
r\in[-m,-1]\cup\hf+\Z_+\}$. The corresponding Dynkin diagram is
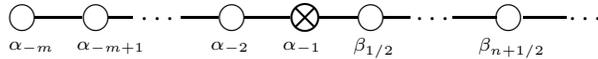
\begin{center}
\hskip -3cm \setlength{\unitlength}{0.16in}
\begin{picture}(24,3)
\put(6,2){\makebox(0,0)[c]{$\bigcirc$}}
\put(8.4,2){\makebox(0,0)[c]{$\bigcirc$}}
\put(12.85,2){\makebox(0,0)[c]{$\bigcirc$}}
\put(15.25,2){\makebox(0,0)[c]{$\bigotimes$}}
\put(17.4,2){\makebox(0,0)[c]{$\bigcirc$}}
\put(21.9,2){\makebox(0,0)[c]{$\bigcirc$}}
\put(6.4,2){\line(1,0){1.55}} \put(8.82,2){\line(1,0){0.8}}
\put(11.2,2){\line(1,0){1.2}} \put(13.28,2){\line(1,0){1.45}}
\put(15.7,2){\line(1,0){1.25}} \put(17.8,2){\line(1,0){0.9}}
\put(20.1,2){\line(1,0){1.4}}
\put(22.4,2){\line(1,0){1.4}}
\put(10.5,1.95){\makebox(0,0)[c]{$\cdots$}}
\put(19.5,1.95){\makebox(0,0)[c]{$\cdots$}}
\put(24.5,1.95){\makebox(0,0)[c]{$\cdots$}}
\put(6.3,1){\makebox(0,0)[c]{\tiny $\alpha_{-m}$}}
\put(8.9,1){\makebox(0,0)[c]{\tiny $\alpha_{-m+1}$}}
\put(12.8,1){\makebox(0,0)[c]{\tiny $\alpha_{-2}$}}
\put(15.15,1){\makebox(0,0)[c]{\tiny $\alpha_{-1}$}}
\put(17.5,1){\makebox(0,0)[c]{\tiny $\beta_{1/2}$}}
\put(22,1){\makebox(0,0)[c]{\tiny $\beta_{n+1/2}$}}
\end{picture}
\end{center}
where
\begin{equation*}
\beta_{i-1/2}:=\epsilon_{i-1/2}-\epsilon_{i+1/2}, \quad i\ge 1.
\end{equation*}

Set $\ov{\mf{b}}:={\wt{\mf{b}}}\cap\SG$,
$\ov{{\mf{n}}}:={\wt{\mf{n}}}\cap\SG$ and
$\ov{\mf{n}}_-:={\wt{\mf{n}}}_-\cap\SG$. Set
$\ov{\mf{l}}_Y:=\wt{\mf{l}}_Y\cap\SG$,
$\ov{\mf{p}}_Y:=\wt{\mf{p}}_Y\cap\SG$,
$\ov{\mf{u}}_Y:=\wt{\mf{u}}_Y\cap\SG$, and
$(\ov{\mf{u}}_Y)_-:=(\wt{\mf{u}}_Y)_-\cap\SG$.

Given $\la\in\ov{\h}^*$ let $L(\ov{\mf{l}}_Y,\la)$ be the
irreducible $\ov{\mf{l}}_Y$-module of highest weight $\la$ with
respect to $\ov{\mf{l}}_Y\cap\ov{\mf{b}}$, which we may regard as
an irreducible $\ov{\mf{p}}_Y$-module. Define the {\em parabolic Verma $\SG$-module}
\begin{equation*}
\ov{K}(\la):={\rm
Ind}_{\ov{\mf{p}}_Y}^{{\SG}}L(\ov{\mf{l}}_Y,\la).
\end{equation*} Let $\ov{L}(\la)$ be the
irreducible ${\SG}$-module of highest weight $\la$ with respect to
$\ov{\mf{b}}$.

 Let
\begin{equation*}
\ov{\mc{P}}_Y:=\{\la^{<0}+\sum_{i\in\N}(\la_+)_i'\epsilon_{i-\hf}\in\ov{\h}^*\mid
\,\la=(\la_i)\in\mc{P}_Y\}.
\end{equation*}

For a semisimple $\ov{\h}$-module $\ov{M}$ and $\gamma\in
\ov{\h}^*$, we define
$$\ov{M}_\gamma:=\{m\in \ov{M}\vert hm=\gamma(h)m,\forall
h\in\ov{\h}\}.$$

\subsection{Parametrization for $\mc{P}_Y$, $\ov{\mc{P}}_Y$ and ${\wt{\mc{P}}_Y}$}
The set $\mc{P}_Y$ parameterizes the sets $\mc{P}_Y$, $\ov{\mc{P}}_Y$ and
${\wt{\mc{P}}_Y}$. From now on we will use the following notation. For
$\la=(\la_i)\in\mc{P}_Y$, let
\begin{align*}
&\la:=\la^{<0}+\sum_{i=1}^\infty{\la}_i\epsilon_i\in \mc{P}_Y,\\
&\la^\natural:=\la^{<0}+\sum_{i=1}^\infty(\la_+)'_i\epsilon_{i-1/2}\in\ov{\mc{P}}_Y,\\
&\la^\theta:=\la^{<0}+\sum_{r\in\hf\N}{\theta(\la_+)}_r\epsilon_r\in{\wt{\mc{P}}_Y}.
\end{align*}
\begin{example} Let $m=3$, $Y=\emptyset$ and $\la=(-5,2,-3,7,6,3,3,1)$ (cf.~\exref{eg:theta}).
We have \begin{align*}
&\la=-5\epsilon_{-3}+2\epsilon_{-2}-3\epsilon_{-1} +7\epsilon_{1}+6\epsilon_{2}+3\epsilon_{3}+3\epsilon_{4}+1\epsilon_{5},\\
&\la^\natural=-5\epsilon_{-3}+2\epsilon_{-2}-3\epsilon_{-1} +5\epsilon_{\frac{1}{2}}+4\epsilon_{\frac{3}{2}}+4\epsilon_{\frac{5}{2}}
+2\epsilon_{\frac{7}{2}}+2\epsilon_{\frac{9}{2}}+2\epsilon_{\frac{11}{2}}
+1\epsilon_{\frac{13}{2}},\\
&\la^\theta=-5\epsilon_{-3}+2\epsilon_{-2}-3\epsilon_{-1} +5\epsilon_{\frac{1}{2}}
+6\epsilon_{1}+3\epsilon_{\frac{3}{2}}+4\epsilon_{2}+2\epsilon_{\frac{5}{2}}.
\end{align*}
\end{example}

\subsection{Categories of $\wt{\G}$-, $\G$-, and $\SG$-modules}\label{sec:O}
Let $\wt{\mc{O}}_Y$ (respectively ${\mc{O}}_Y$, $\ov{\mc{O}}_Y$) be the category of
$\wt{\G}$-(respectively $\G$-, $\SG$-)modules ${M}$ such that ${M}$ is a semisimple
$\wt{\h}$-(respectively $\h$-, $\ov{\h}$-)module and ${\rm dim}{M}_\gamma<\infty$ for
each $\gamma\in \wt{\h}^*$ (respectively ${\h}^*$, $\ov{\h}^*$) satisfying
\begin{itemize}
\item[(i)] ${M}$ decomposes over $\wt{\mf{l}}_Y$ (respectively ${\mf{l}}_Y$,
    $\ov{\mf{l}}_Y$) into a direct sum of $L(\wt{\mf{l}}_Y,\mu^\theta)$ (respectively
    $L({\mf{l}}_Y,\mu)$, $L(\ov{\mf{l}}_Y,\mu^\natural)$), $\mu\in{\mc{P}}_Y$,
\item[(ii)] ${M}$ has a filtration of $\wt{\G}$-(respectively $\G$-, $\SG$-)modules
    ${M}={M}_0\supseteq {M}_1\supseteq{M}_2\supseteq\cdots$ such that for all $i\ge
    0$ ${M}_i/{M}_{i+1}\cong \wt{L}({\nu}^\theta_i)$ (respectively ${L}(\nu_i)$,
    $\ov{L}({\nu}^\natural_i)$), for some $\nu_i\in{\mc{P}}_Y$.
\end{itemize}

The morphisms in $\mc{O}_Y$ are $\G$-homomorphisms. The morphisms in $\wt{\mc{O}}_Y$
(respectively $\ov{\mc{O}}_Y$) are (not necessarily even) $\wt{\G}$-(respectively
$\SG$-)homomorphisms. Clearly ${K}(\la)$, for $\la\in{{\mc{P}}_Y}$, lies in ${\mc{O}}_Y$.
 By \cite[Theorem 3.2]{CK} and \cite[Theorem 3.1]{CK}, $\wt{K}(\la^\theta)$ and $\ov{K}(\la^\natural)$,
for $\la\in{\mc{P}}_Y$, lie in $\wt{\mc{O}}_Y$ and $\ov{\mc{O}}_Y$, respectively.

Let $\wt{\mc{O}}^{f}_Y$ (respectively ${\mc{O}}^{f}_Y$, $\ov{\mc{O}}^{f}_Y$) denote the
full subcategory of $\wt{\mc{O}}_Y$ (respectively ${\mc{O}}_Y$, $\ov{\mc{O}}_Y$)
consisting of objects having finite composition series.

Set $\wt{\Gamma}:=\sum_{j=-m}^{-1}\Z\epsilon_j+\sum_{r\in\hf\N}\Z\epsilon_r$. Let $\wt{V}$ be a semisimple $\wt{\h}$-module  such that
$\wt{M}=\bigoplus_{\gamma\in\wt{\Gamma}}\wt{M}_\gamma$. Then $\wt{V}$ is a
$\Z_2$-graded vector space $\wt{V}=\wt{V}_{\bar{0}}\bigoplus
\wt{V}_{\bar{1}}$ such that
 \begin{equation}\label{wt-Z2-gradation}
 \wt{V}_{\bar{0}}:=\bigoplus_{\mu\in\wt{\Gamma}_{\bar{0}}}\wt{V}_{\mu}\qquad\hbox{and}\qquad
 \wt{V}_{\bar{1}}:=\bigoplus_{\mu\in\wt{\Gamma}_{\bar{1}}}\wt{V}_{\mu},
 \end{equation}
where $\wt{\Gamma}_\epsilon:=\{\mu\in \wt{\Gamma}\mid \sum_{r\in
{1/ 2}+\Z_+}\mu(E_{r,r})\equiv {\epsilon} \,\,(\text{mod }2)\}$.

Set $\ov{\Gamma}:=\sum_{j=-m}^{-1}\Z\epsilon_j+\sum_{r\in\hf+\Z_+}\Z\epsilon_r$. Let $\ov{V}$ be a semisimple $\ov{\h}$-module such that
$\ov{M}=\bigoplus_{\gamma\in\ov{\Gamma}}\ov{M}_\gamma$. Then $\ov{V}$ is an
$\Z_2$-graded vector space $\ov{V}=\ov{V}_{\bar{0}}\bigoplus
\ov{V}_{\bar{1}}$ such that
 \begin{equation}\label{ov-Z2-gradation}
 \ov{V}_{\bar{0}}:=\bigoplus_{\mu\in\ov{\Gamma}_{\bar{0}}}\ov{V}_{\mu}\qquad\hbox{and}\qquad
 \ov{V}_{\bar{1}}:=\bigoplus_{\mu\in\ov{\Gamma}_{\bar{1}}}\ov{V}_{\mu},
 \end{equation}
where $\ov{\Gamma}_\epsilon:=\{\mu\in \ov{\Gamma}\mid \sum_{r\in
{1\over 2}+\Z_+}\mu(E_{r,r})\equiv {\epsilon} \,\,(\text{mod }2)\}$.

It is clear that ${\mc{O}}_Y$ is an abelian category. For $\wt{M}\in\wt{\mc{O}}_Y$, let
$\widehat{\wt{M}}\in\wt{\mc{O}}_Y$ denote the $\wt{\G}$-module $\wt{M}$ equipped with the
$\Z_2$-gradation given by \eqnref{wt-Z2-gradation}. The $\Z_2$-gradation given by
\eqnref{ov-Z2-gradation} is compatible with the $\wt{\G}$-action. Therefore, for
$\wt{M}$, $\wt{N}\in\wt{\mc{O}}_Y$, and $\wt{\varphi}\in{\rm Hom}_{\wt{\mc{O}}_Y}(\wt{M},
\wt{N})$, the kernel and the cokernel of $\wt{\varphi}$ have structures of $\Z_2$-graded
vector spaces defined by \eqnref{wt-Z2-gradation}. The induced $\wt{\G}$-actions on the
kernel and cokernel are compatible with this $\Z_2$-gradation. Thus the kernel and the
cokernel of $\wt{\varphi}$ belong to $\wt{\mc{O}}_Y$, and hence $\wt{\mc{O}}_Y$ and
$\wt{\mc{O}}_Y^f$ are abelian categories. Note that the homomorphic image of
$\wt{\varphi}$ may not be a $\wt{\G}$-submodule of $\wt{N}$. Similarly, the
$\Z_2$-gradation given by \eqnref{ov-Z2-gradation} of any object
${\ov{M}}\in\ov{\mc{O}}_Y$ is compatible with its $\ov{\G}$-action. Moreover,
$\ov{\mc{O}}_Y$ and $\ov{\mc{O}}_Y^f$ are abelian categories.

We define $\wt{\mc{O}}^{\bar{0}}_Y$ and $\wt{\mc{O}}^{f,{\bar{0}}}_Y$ to be the full
subcategories of $\wt{\mc{O}}_Y$ and $\wt{\mc{O}}_Y^f$,
respectively, consisting
of objects with $\Z_2$-gradations given by \eqnref{wt-Z2-gradation} (c.f.~\cite[\S4-e]{B}). Note that the morphisms in $\wt{\mc{O}}^{\bar{0}}_Y$ and
$\wt{\mc{O}}^{f, {\bar{0}}}_Y$ are of degree $\bar{0}$. For $\wt{M}\in\wt{\mc{O}}_Y$,
it is clear that
$\widehat{\wt{M}}$ is
isomorphic to $\wt{M}$ in $\wt{\mc{O}}_Y$. Thus
$\wt{\mc{O}}_Y$ and $\wt{\mc{O}}^{\bar{0}}_Y$ have isomorphic skeletons and hence they are equivalent categories.  Similarly, $\wt{\mc{O}}^{f}_Y$ and $\wt{\mc{O}}_Y^{f, {\bar{0}}}$ are equivalent
categories.

Analogously define $\ov{\mc{O}}^{\bar{0}}_Y$ and $\ov{\mc{O}}^{f,{\bar{0}}}_Y$ to be the respective full
subcategories of $\ov{\mc{O}}_Y$ and $\ov{\mc{O}}_Y^f$ consisting
of objects with $\Z_2$-gradations given by \eqnref{ov-Z2-gradation}. Similarly the morphisms in $\ov{\mc{O}}^{\bar{0}}_Y$ and
$\ov{\mc{O}}^{f,{\bar{0}}}_Y$ are of degree $\bar{0}$. Moreover, $\ov{\mc{O}}^{\bar{0}}_Y$ and $\ov{\mc{O}}_Y$ are equivalent categories, and $\ov{\mc{O}}^{f, {\bar{0}}}_Y$ and $\ov{\mc{O}}_Y^f$ are equivalent
categories.

\section{Odd Reflection and character formulae}\label{section3}
We shall briefly explain the effect of an odd reflection on the
highest weight of an irreducible module that was studied in
\cite[Lemma 1]{PS} (see also \cite[Lemma 1.4]{KW2}). Fix a Borel subalgebra $\mc{B}$ with corresponding set of positive roots $\Delta_+(\mc{B})$. Let $\alpha$
be an isotropic odd simple root and $\alpha^\vee$ be its
corresponding coroot. Applying the odd reflection with respect to
$\alpha$ changes the Borel subalgebra $\mc{B}$ into a new Borel subalgebra $\mc{B}(\alpha)$ with corresponding set of positive roots $\Delta_+(\mc{B}(\alpha))=\{-\alpha\}\cup\Delta_+(\mc{B})\setminus\{\alpha\}$.
Now let $\la$ be the highest weight with respect to $\mc{B}$ of an irreducible module.  If
$\langle\la,\alpha^\vee\rangle\not=0$, then the highest weight of
this irreducible module with respect to $\mc{B}(\alpha)$
is $\la-\alpha$.  If $\langle\la,\alpha^\vee\rangle=0$, then the
highest weight remains unchanged.  In the sequel we will sometimes
refer to the highest weight with respect to the new Borel
subalgebra as the {\em new highest weight}.

\subsection{Odd reflection and a fundamental lemma}

\subsubsection{A sequence of odd reflections and $\wt{\Pi}^c(n)$}\label{afterkk}
Starting with the Dynkin diagram (D1) of \secref{aux122} and given a positive integer $n$
we apply the following sequence of $\frac{n(n+1)}{2}$ odd reflections. First we apply one
odd reflection corresponding to $\epsilon_{1/2}-\epsilon_1$, then we apply two odd
reflections corresponding to $\epsilon_{3/2}-\epsilon_{2}$ and
$\epsilon_{1/2}-\epsilon_{2}$.  After that we apply three odd reflections corresponding
to $\epsilon_{5/2}-\epsilon_{3}$, $\epsilon_{3/2}-\epsilon_{3}$, and
$\epsilon_{1/2}-\epsilon_{3}$, et cetera, until finally we apply $n$ odd reflections
corresponding to
$\epsilon_{n-1/2}-\epsilon_{n},\epsilon_{n-3/2}-\epsilon_{n},\cdots,\epsilon_{1/2}-\epsilon_{n}$.
The resulting new Borel subalgebra for $\DG$ will be denoted by $\wt{\mf{b}}^c(n)$ and
the corresponding simple roots are
\begin{equation*}
\wt{\Pi}^c(n):=\{\alpha_{-m},\cdots,\alpha_{-2};\beta_{-1};\beta_1,\cdots,\beta_{n-1};\epsilon_{n}-\epsilon_{1/2};
\beta_{1/2},\cdots,
\beta_{n-1/2};\alpha_{n+1/2},\alpha_{n+1},\cdots\}.
\end{equation*}
\begin{center}
\hskip -3cm \setlength{\unitlength}{0.16in}
\begin{picture}(24,3)
\put(0,2){\makebox(0,0)[c]{$\bigcirc$}} \put(2.4,2){\makebox(0,0)[c]{$\bigcirc$}}
\put(6,2){\makebox(0,0)[c]{$\bigcirc$}} \put(8.4,2){\makebox(0,0)[c]{$\bigcirc$}}
\put(12.85,2){\makebox(0,0)[c]{$\bigcirc$}} \put(15.25,2){\makebox(0,0)[c]{$\bigotimes$}}
\put(17.4,2){\makebox(0,0)[c]{$\bigcirc$}} \put(21.9,2){\makebox(0,0)[c]{$\bigcirc$}}
\put(24.4,2){\makebox(0,0)[c]{$\bigotimes$}} \put(26.8,2){\makebox(0,0)[c]{$\bigotimes$}}
\put(0.4,2){\line(1,0){1.55}} \put(2.8,2){\line(1,0){1}} \put(6.4,2){\line(1,0){1.55}}
\put(8.82,2){\line(1,0){0.8}}
\put(11.2,2){\line(1,0){1.2}} \put(13.28,2){\line(1,0){1.45}}
\put(15.7,2){\line(1,0){1.25}} \put(17.8,2){\line(1,0){0.9}}
\put(20.1,2){\line(1,0){1.4}} \put(22.35,2){\line(1,0){1.6}}
\put(24.9,2){\line(1,0){1.5}} \put(27.3,2){\line(1,0){1.5}}
\put(4.7,1.95){\makebox(0,0)[c]{$\cdots$}} \put(10.5,1.95){\makebox(0,0)[c]{$\cdots$}}
\put(19.5,1.95){\makebox(0,0)[c]{$\cdots$}} \put(29.7,1.95){\makebox(0,0)[c]{$\cdots$}}
\put(0,1){\makebox(0,0)[c]{\tiny $\alpha_{-m}$}} \put(2.3,1){\makebox(0,0)[c]{\tiny
$\alpha_{-m+1}$}} \put(6.1,1){\makebox(0,0)[c]{\tiny $\alpha_{-2}$}}
\put(8.5,1){\makebox(0,0)[c]{\tiny $\beta_{-1}$}} \put(12.8,1){\makebox(0,0)[c]{\tiny
$\beta_{n-1}$}} \put(15.2,1){\makebox(0,0)[c]{\tiny $\epsilon_n-\epsilon_{1/2}$}}
\put(17.8,1){\makebox(0,0)[c]{\tiny $\beta_{1/2}$}} \put(22,1){\makebox(0,0)[c]{\tiny
$\beta_{n-1/2}$}} \put(24.4,1){\makebox(0,0)[c]{\tiny $\alpha_{n+1/2}$}}
\put(27,1){\makebox(0,0)[c]{\tiny $\alpha_{n+1}$}}
\end{picture}
\end{center}
A Borel subalgebra is completely determined by giving an ordered homogeneous basis for
the standard module. The ordered basis corresponding to the Borel subalgebra of
$\wt{\Pi}^c(n)$ is given below.
\begin{align*}
\{v_{-m},\ldots, v_{-1}, v_1, v_2, \ldots, v_{n-1}, v_n, v_{1/2}, v_{3/2},\ldots,
v_{n+1/2}, v_{n+1}, v_{n+3/2},v_{n+2}, v_{n+5/2},\ldots\}.
\end{align*}

\subsubsection{A sequence of odd reflections and $\wt{\Pi}^s(n)$}
On the other hand given (D1) and $n$ we can also apply the following different sequence
of $\frac{n(n+1)}{2}$ odd reflections. First we apply one odd reflection corresponding to
$\epsilon_{1}-\epsilon_{3/2}$, then we apply two odd reflections corresponding to
$\epsilon_{2}-\epsilon_{5/2}$ and $\epsilon_{1}-\epsilon_{5/2}$.  After that we apply
three odd reflections corresponding to $\epsilon_{3}-\epsilon_{7/2}$,
$\epsilon_{2}-\epsilon_{7/2}$, and $\epsilon_{1}-\epsilon_{7/2}$, et cetera, until
finally we apply $n$ odd reflections corresponding to
$\epsilon_{n}-\epsilon_{n+1/2},\epsilon_{n-1}-\epsilon_{n+1/2},\cdots,\epsilon_{1}-\epsilon_{n+1/2}$.
The resulting new Borel subalgebra for $\DG$ will be denoted by $\wt{\mf{b}}^s(n)$ and
the corresponding simple roots are
\begin{equation*}
\wt{\Pi}^s(n):=\{\alpha_{-m},\cdots,\alpha_{-2};\alpha_{-1};\beta_{1/2},\cdots
\beta_{n-1/2};\epsilon_{n+1/2}-\epsilon_{1};\beta_1,\cdots,\beta_{n};
\alpha_{n+1},\alpha_{n+3/2},\cdots\}.
\end{equation*}
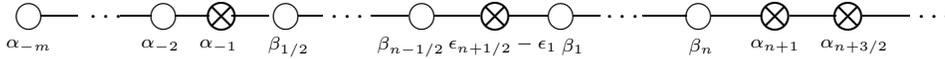
\begin{center}
\hskip -3cm \setlength{\unitlength}{0.16in}
\begin{picture}(24,3)
\put(0,2){\makebox(0,0)[c]{$\bigcirc$}} \put(4.4,2){\makebox(0,0)[c]{$\bigcirc$}}
\put(6.3,2){\makebox(0,0)[c]{$\bigotimes$}} \put(8.4,2){\makebox(0,0)[c]{$\bigcirc$}}
\put(12.85,2){\makebox(0,0)[c]{$\bigcirc$}} \put(15.25,2){\makebox(0,0)[c]{$\bigotimes$}}
\put(17.4,2){\makebox(0,0)[c]{$\bigcirc$}} \put(21.9,2){\makebox(0,0)[c]{$\bigcirc$}}
\put(24.4,2){\makebox(0,0)[c]{$\bigotimes$}} \put(26.8,2){\makebox(0,0)[c]{$\bigotimes$}}
\put(0.4,2){\line(1,0){1.2}} \put(3,2){\line(1,0){1}} \put(4.8,2){\line(1,0){1.1}}
\put(6.65,2){\line(1,0){1.2}} \put(8.82,2){\line(1,0){0.8}}
\put(11.2,2){\line(1,0){1.2}} \put(13.28,2){\line(1,0){1.45}}
\put(15.7,2){\line(1,0){1.25}} \put(17.8,2){\line(1,0){0.9}}
\put(20.1,2){\line(1,0){1.4}} \put(22.35,2){\line(1,0){1.6}}
\put(24.9,2){\line(1,0){1.5}} \put(27.3,2){\line(1,0){1.5}}
\put(2.6,1.95){\makebox(0,0)[c]{$\cdots$}} \put(10.5,1.95){\makebox(0,0)[c]{$\cdots$}}
\put(19.5,1.95){\makebox(0,0)[c]{$\cdots$}} \put(29.7,1.95){\makebox(0,0)[c]{$\cdots$}}
\put(0,1){\makebox(0,0)[c]{\tiny $\alpha_{-m}$}} \put(4.3,1){\makebox(0,0)[c]{\tiny
$\alpha_{-2}$}} \put(6.2,1){\makebox(0,0)[c]{\tiny $\alpha_{-1}$}}
\put(8.5,1){\makebox(0,0)[c]{\tiny $\beta_{1/2}$}} \put(12.5,1){\makebox(0,0)[c]{\tiny
$\beta_{n-1/2}$}} \put(15.5,1){\makebox(0,0)[c]{\tiny $\epsilon_{n+1/2}-\epsilon_{1}$}}
\put(17.8,1){\makebox(0,0)[c]{\tiny $\beta_{1}$}} \put(22,1){\makebox(0,0)[c]{\tiny
$\beta_{n}$}} \put(24.4,1){\makebox(0,0)[c]{\tiny $\alpha_{n+1}$}}
\put(27,1){\makebox(0,0)[c]{\tiny $\alpha_{n+3/2}$}}
\end{picture}
\end{center}
The ordered basis corresponding to the Borel subalgebra of $\wt{\Pi}^s(n)$ is given
below.
\begin{align*}
\{v_{-m},\ldots,v_{-1}, v_{1/2}, v_{3/2},\ldots,v_{n-1/2}, v_{n+1/2}, v_1,
v_2, \ldots, v_n, v_{n+1}, v_{n+3/2},v_{n+2}, v_{n+5/2},\ldots\}.
\end{align*}

\begin{rem}\label{aux:rem}
We note that the simple roots used in the above two sequences of odd reflections are all
roots of $\wt{\mf{l}}_Y$ and hence these sequences of odd reflections leave the set of
roots of $\wt{\mf{u}}_Y$ invariant.
\end{rem}

We denote by ${\wt{\mf{b}}_Y^c}(n)$ and ${\wt{\mf{b}}_Y^s}(n)$ the Borel subalgebras of
$\wt{\mf{l}}_Y$ corresponding to the sets of simple roots
$\wt{\Pi}^c(n)\cap\wt{\Delta}_Y$ and $\wt{\Pi}^s(n)\cap\wt{\Delta}_Y$, respectively.

\subsubsection{A fundamental lemma}

\begin{lem}\label{lem:change} Given $\la\in\mc{P}_Y$, let $n\in\N$.
\begin{itemize}
\item[(i)] Suppose that $\ell(\la_+)\le n$. Then the highest weight of
$L(\wt{\mf{l}}_Y,\la^\theta)$ with respect to the Borel subalgebra
$\wt{\mf{b}}_Y^c(n)$ is $\la$, regarded as an element in ${\wt{\mc{P}}_Y}$.
\item[(ii)] Suppose that $\ell(\la_+')\le n$. Then the highest weight of
    $L(\wt{\mf{l}}_Y,\la^\theta)$ with respect to the Borel subalgebra
    $\wt{\mf{b}}_Y^s(n)$ is $\la^\natural$, regarded as an element in
    ${\wt{\mc{P}}_Y}$.
\end{itemize}
\end{lem}

\begin{proof} We shall only give the proof for (i), as (ii) is analogous.

Certainly $\la^{<0}$ is unaffected by the sequence of odd
reflections in \secref{afterkk}. We will show more generally by induction on $k$ that after applying the
first $k(k+1)/2$ odd reflections in \secref{afterkk} this weight
becomes
\begin{equation}\label{afterk}
\la_{[k]}=\la^{<0}+\sum_{i=1}^k\la_i\epsilon_i+\sum_{i=1}^k\langle(\la_+)'_i-k\rangle\epsilon_{i-1/2}+
\sum_{j\ge k+1}\langle(\la_+)'_{j}-j+1\rangle\epsilon_{j-1/2} +
\sum_{j\ge k+1}\langle\la_j-j\rangle\epsilon_j.
\end{equation}
From \eqnref{afterk} the lemma follows.

Suppose that $k=1$.
If $\ell(\la_+)< 1$, then $\la^\theta=\la^{<0}$ and in particular
$\langle\la^\theta,E_{1/2,1/2}+E_{11}\rangle=0$, and thus the
new highest weight is $\la_{[1]}=\la^{<0}=\la^\theta$.  If
$\ell(\la_+)\ge 1$, then $(\la_+)_1'\ge 1$ and $\la_1\ge 1$ and
thus
\begin{equation*}
\la^\theta=\la^{<0}+(\la_+)_1'\epsilon_{1/2}+(\la_1-1)\epsilon_1+\cdots.
\end{equation*}
Now $\langle\la^\theta,E_{1/2,1/2}+E_{11}\rangle>0$, and hence the
highest weight after the odd reflection with respect to
$\epsilon_{1/2}-\epsilon_1$ is
\begin{equation*}
\la_{[1]}=\la^{<0}+\la_1\epsilon_1+((\la_+)_1'-1)\epsilon_{1/2}+\cdots,
\end{equation*}
proving \eqnref{afterk} in the case $k=1$.

Now suppose that \eqnref{afterk} is true for $k$.  We shall derive
the formula for $k+1$. If $\ell(\la_+)\le k$, then
$\la_{[k]}=\la^{<0}+\sum_{i=1}^k\la_i\epsilon_i$. Therefore we
have $\langle\la_{[k]},E_{i-1/2,i-1/2}+E_{k+1,k+1}\rangle=0$,
for $1\le i\le k+1$. So the odd reflections with respect to
$\epsilon_{i-1/2}-\epsilon_{k+1}$ do not affect $\la_{[k]}$.
Thus we have $\la_{[k+1]}=\la_{[k]}$. So in this case we are
done.

Now assume that $\ell(\la_+)\ge k+1$. Let $s=\la_{k+1}$. We
distinguish two cases.

First suppose that $\la_{k+1}\ge k+1$.  Then $(\la_+)'_{k+1}\ge
k+1$ and hence \eqnref{afterk} becomes
\begin{align*}
\la_{[k]}=&\la^{<0}+\sum_{i=1}^k\la_i\epsilon_i
+\sum_{i=1}^k((\la_+)'_i-k)\epsilon_{i-1/2} +
((\la_+)'_{k+1}-k)\epsilon_{k+1/2}\\ &+
(\la_{k+1}-k-1)\epsilon_{k+1} + \sum_{j\ge
k+2}\langle(\la_+)'_{j}-j+1\rangle\epsilon_{j-1/2} + \sum_{j\ge
k+2}\langle\la_j-j\rangle\epsilon_j.
\end{align*}
Now $\langle\la_{[k]},E_{k+1/2,k+1/2}+E_{k+1,k+1}\rangle>0$ so
that after the odd reflection with respect to
$\epsilon_{k+1/2}-\epsilon_{k+1}$ the new weight becomes
\begin{align*}
\la_{[k,1]}=&\la^{<0}+\sum_{i=1}^k\la_i\epsilon_i
+\sum_{i=1}^k((\la_+)'_i-k)\epsilon_{i-1/2} +
((\la_+)'_{k+1}-k-1)\epsilon_{k+1/2}\\& +
(\la_{k+1}-k)\epsilon_{k+1} + \sum_{j\ge
k+2}\langle(\la_+)'_{j}-j+1\rangle\epsilon_{j-1/2} + \sum_{j\ge
k+2}\langle\la_j-j\rangle\epsilon_j.
\end{align*}
Now $\langle\la_{[k,1]},E_{k-1/2,k-1/2}+E_{k+1,k+1}\rangle>0$ so
after the odd reflection with respect to
$\epsilon_{k-1/2}-\epsilon_{k+1}$ we get
\begin{align*}
\la_{[k,2]}=&\la^{<0}+\sum_{i=1}^k\la_i\epsilon_i+
\sum_{i=1}^{k-1}((\la_+)'_i-k)\epsilon_{i-1/2}
+((\la_+)'_{k}-k-1)\epsilon_{k-1/2} \\
&+((\la_+)'_{k+1}-k-1)\epsilon_{k+1/2}
+(\la_{k+1}-k+1)\epsilon_{k+1}\\& + \sum_{j\ge
k+2}\langle(\la_+)'_{j}-j+1\rangle\epsilon_{j-1/2} + \sum_{j\ge
k+2}\langle\la_j-j\rangle\epsilon_j.
\end{align*}
Finally after a total of $k+1$ odd reflections we end up with
\begin{align*}
\la_{[k,k+1]}=&\la^{<0}+\sum_{i=1}^{k+1}\la_i\epsilon_i+\sum_{i=1}^{k+1}((\la_+)'_i-k-1)\epsilon_{i-1/2}\\
&+ \sum_{j\ge k+2}\langle(\la_+)'_{j}-j+1\rangle\epsilon_{j-1/2} +
\sum_{j\ge k+2}\langle\la_j-j\rangle\epsilon_j,
\end{align*}
which equals $\la_{[k+1]}$.

Now consider the case $\la_{k+1}=s<k+1$. We have $(\la_+)'_j\ge
k+1$, for $j\le s$ and $(\la_+)'_j<k+1$, for $j>s$. Thus
\eqnref{afterk} becomes
\begin{align*}
\la_{[k]}=\la^{<0}+\sum_{i=1}^k\la_i\epsilon_i+\sum_{i=1}^s((\la_+)'_i-k)\epsilon_{i-1/2},
\end{align*}
where $((\la_+)'_i-k)>0$, for $i\le s$. It follows that odd
reflections with respect to
$\epsilon_{k+1/2}-\epsilon_{k+1},\cdots,\epsilon_{s+1/2}-\epsilon_{k+1}$
do not affect $\la_{[k]}$, while odd refections with respect to
$\epsilon_{s-1/2}-\epsilon_{k+1},\cdots,\epsilon_{1/2}-\epsilon_{k+1}$
affect $\la_{[k]}$. From this we obtain
\begin{align*}
\la_{[k+1]}=\la_{[k,k+1]}=\la^{<0}+\sum_{i=1}^{k+1}\la_i\epsilon_i+\sum_{i=1}^s((\la_+)'_i-k-1)\epsilon_{i-1/2},
\end{align*}
which concludes the proof.
\end{proof}

\begin{cor}\label{aux:cor1} Let $\la\in\mc{P}_Y$ and $n\in\N$.
\begin{itemize}
\item[(i)] Suppose that $\ell(\la_+)\le n$. Then $\wt{K}(\la^\theta)$ is a highest
    weight module with respect to the Borel subalgebra $\wt{\mf{b}}^c(n)$ with
    highest weight $\la$, regarded as an element in ${\wt{\mc{P}}_Y}$. Also the
    highest weight of $\wt{L}(\la^\theta)$ with respect to the Borel subalgebra
    $\wt{\mf{b}}^c(n)$ is $\la$, regarded as an element in ${\wt{\mc{P}}_Y}$.
\item[(ii)] Suppose that $\ell(\la_+')\le n$. Then $\wt{K}(\la^\theta)$ is a highest
    weight module with respect to the Borel subalgebra $\wt{\mf{b}}^s(n)$ with
    highest weight $\la^\natural$, regarded as an element in ${\wt{\mc{P}}_Y}$. Also
    the highest weight of $\wt{L}(\la^\theta)$ with respect to the Borel subalgebra
    $\wt{\mf{b}}^s(n)$ is $\la^\natural$, regarded as an element in
    ${\wt{\mc{P}}_Y}$.
\end{itemize}
\end{cor}

\begin{proof}
As an $\wt{\mf{l}}_Y$-module $\wt{K}(\la^\theta)$ contains a unique copy of
$L(\wt{\mf{l}}_Y,\la^\theta)$ that is annihilated by $\wt{\mf{u}}_Y$.  By
\lemref{lem:change} with respect to $\wt{\mf{b}}_Y^c(n)$ the highest weight of
$L(\wt{\mf{l}}_Y,\la^\theta)$ is $\la$. Now by \remref{aux:rem}
$\wt{\mf{b}}_Y^c(n)+\wt{\mf{u}}_Y=\wt{\mf{b}}^c(n)$. Thus $\wt{K}(\la^\theta)$ has a
non-zero vector of weight $\la$ annihilated by $\wt{\mf{b}}^c(n)$. This vector clearly
generates $\wt{K}(\la^\theta)$ over $\DG$, proving the first statement of (i).  A
verbatim argument proves the second statement as well.

Part (ii) is similar and so its proof is omitted.
\end{proof}

\subsection{The Functors $T$ and $\ov{T}$}\label{Tfunctors} Recall that
$\wt{\Gamma}:=\sum_{j=-m}^{-1}\Z\epsilon_j+\sum_{r\in\hf\N}\Z\epsilon_r$ and
$\ov{\Gamma}:=\sum_{j=-m}^{-1}\Z\epsilon_j+\sum_{r\in\hf+\Z_+}\Z\epsilon_r$. Set
$\Gamma:=\sum_{j=-m}^{-1}\Z\epsilon_j+\sum_{i\in\N}\Z\epsilon_i$. Given a semisimple
$\wt{\h}$-module $\wt{M}$ such that
$\wt{M}=\bigoplus_{\gamma\in\wt{\Gamma}}\wt{M}_\gamma$, we define
\begin{align*}
T(\wt{M}):= \bigoplus_{\gamma\in\Gamma}\wt{M}_\gamma,\qquad
\hbox{and}\qquad \ov{T}(\wt{M}):=
\bigoplus_{\gamma\in\ov{\Gamma}}\wt{M}_\gamma.
\end{align*}
Note that $T(\wt{M})$ is an $\h$-submodule of $\wt{M}$ (regarded as an $\h$-module), and
$\ov{T}(\wt{M})$ is an $\ov{\h}$-submodule of $\wt{M}$ (regarded as an $\ov{\h}$-module).
Also if $\wt{M}$ is also an $\wt{\mf{l}}_Y$-module, then $T(\wt{M})$ is an
${\mf{l}}_Y$-submodule of $\wt{M}$ (regarded as an ${\mf{l}}_Y$-module), and
$\ov{T}(\wt{M})$ is an $\ov{\mf{l}}_Y$-submodule of $\wt{M}$ (regarded as an
$\ov{\mf{l}}_Y$-module). Furthermore if $\wt{M}\in\wt{\mc{O}}_Y$, then $T(\wt{M})$ is a
$\G$-submodule of $\wt{M}$ (regarded as a $\G$-module), and $\ov{T}(\wt{M})$ is a
$\SG$-submodule of $\wt{M}$ (regarded as a $\SG$-module).

Let $\wt{M}=\bigoplus_{\gamma\in\wt{\Gamma}}\wt{M}_\gamma$ and
$\wt{N}=\bigoplus_{\gamma\in\wt{\Gamma}}\wt{N}_\gamma$ be two
semisimple $\wt{\h}$-modules. We let
\begin{eqnarray*}
\CD
 T_{\wt{M}} :  \wt{M} @>>>T(\wt{M}) \qquad \hbox{and}\qquad  \ov{T}_{\wt{M}} :  \wt{M} @>>>\ov{T}(\wt{M})
 \endCD
\end{eqnarray*}
 be the natural projections. If $\wt{f}:\wt{M}\longrightarrow
\wt{N}$ is an $\wt{\h}$-homomorphism, we let \begin{eqnarray*} \CD
 T[\wt{f}] :  T(\wt{M}) @>>>T(\wt{N}) \qquad \hbox{and}\qquad  \ov{T}[\wt{f}] :  \ov{T}(\wt{M}) @>>>\ov{T}(\wt{N})
 \endCD
\end{eqnarray*}
be the corresponding restriction maps. Note that $T_{\wt{M}}$ and $T[\wt{f}]$
(respectively, $\ov{T}_{\wt{M}}$ and $\ov{T}[\wt{f}]$) are $\h$- (respectively,
$\ov{\h}$-) homomorphisms. Also if $\wt{f}$ is also $\wt{\mf{l}}_Y$-homomorphism of
$\wt{\mf{l}}_Y$-modules, then $T_{\wt{M}}$ and $T[\wt{f}]$ (respectively,
$\ov{T}_{\wt{M}}$ and $\ov{T}[\wt{f}]$) are ${\mf{l}}_Y$- (respectively,
$\ov{\mf{l}}_Y$-) homomorphisms. Furthermore if $\wt{f}$ is also $\wt{\G}$-homomorphism
of $\wt{\G}$-modules, then $T_{\wt{M}}$ and $T[\wt{f}]$ (respectively, $\ov{T}_{\wt{M}}$
and $\ov{T}[\wt{f}]$) are $\G$- (respectively, $\SG$-) homomorphisms. It is easy to see
that we have the following commutative diagrams.
\begin{eqnarray}\label{T}
\CD
  \wt{M} @>\wt{f}>>\wt{N} @. \qquad\qquad \wt{M} @>\wt{f}>>\wt{N} \\
  @VVT_{\wt{M}}V @VVT_{\wt{N}}V  \qquad\qquad @VV\ov{T}_{\wt{M}}V @VV\ov{T}_{\wt{N}}V\\
 T(\wt{M}) @>T[\wt{f}]>>T(\wt{N})@. \qquad\qquad \ov{T}(\wt{M}) @>\ov{T}[\wt{f}]>>\ov{T}(\wt{N})\\
 \endCD
\end{eqnarray}

For an indeterminate $e$ we let $x_r:=e^{\epsilon_r}$,
$r\in[-m,-1]\cup\hf\N$.  The formal character of an object in
$\wt{\mc{O}}_Y$, $\ov{\mc{O}}_Y$, and $\mc{O}_Y$ is then an
element in $\Z[[x_{-m}^{\pm 1},\cdots,x_{-1}^{\pm
1}]]\otimes\Z[[x_{1/2},x_1,\cdots]]$, $\Z[[x_{-m}^{\pm
1},\cdots,x_{-1}^{\pm 1}]]\otimes\Z[[x_{1/2},x_{3/2},\cdots]]$,
and $\Z[[x_{-m}^{\pm 1},\cdots,x_{-1}^{\pm
1}]]\otimes\Z[[x_1,x_2,\cdots]]$, respectively. For $\la\in
\mc{P}_Y$, we remark that the character of
$L(\wt{\mf{l}}_Y,\la^\theta)$ is given by \cite[Section 3.2.3]{CK}
\begin{equation}\label{eqn:hookschur}
{\rm ch}L(\wt{\mf{l}}_Y,\la^\theta)={\rm
ch}L(\wt{\mf{l}}^{<0}_Y,\la^{<0})HS_{\la_+'}(x_{1/2},x_1,x_{3/2},x_2,\cdots),
\end{equation}
where (c.f.~\cite{S, BR})
\begin{equation*}
HS_{\eta}(x_{1/2},x_1,x_{3/2},x_2,\cdots):=\sum_{\mu\subseteq\eta}s_{\mu}(x_{1/2},x_{3/2},\cdots)
s_{(\eta/\mu)'}(x_1,x_2,\cdots),\quad\eta\in\mc{P}.
\end{equation*}
Here and below $L(\wt{\mf{l}}^{<0}_Y,\la^{<0})$ stands for the irreducible
$\wt{\mf{l}}_Y^{<0}$-module of highest weight $\la^{<0}$ and so
${\rm ch}L(\wt{\mf{l}}^{<0}_Y,\la^{<0})$ is a product of Schur
Laurent polynomials in $x_{-m},\cdots,x_{-1}$ depending on $Y$ and
$\la^{<0}$.

\begin{lem}\label{lem:aux1} For $\la\in\mc{P}_Y$, we have
\begin{itemize}
\item[(i)] $T(L(\wt{\mf{l}}_Y,\la^\theta))=L(\mf{l}_Y,\la)$,
\item[(ii)] $\ov{T}(L(\wt{\mf{l}}_Y,\la^\theta))=L(\ov{\mf{l}}_Y,\la^\natural)$.
\end{itemize}
\end{lem}

\begin{proof} Since $L(\mf{l}_Y,\la)$ is irreducible,
it is enough to prove that $T(L(\wt{\mf{l}}_Y,\la^\theta))$ and
$L(\mf{l}_Y,\la)$ have the same character.

Applying $T$ to $L(\wt{\mf{l}}_Y,\la^\theta)$ has the effect of
setting $x_{j-{1/2}}=0$, $j\in\N$, in the character.  Thus we have
\begin{equation*}
{\rm ch}T\big{(}L(\wt{\mf{l}}_Y,\la^\theta)\big{)} ={\rm
ch}L(\wt{\mf{l}}^{<0}_Y,\la^{<0})s_{\la_+}(x_1,x_2,\cdots),
\end{equation*}
which is the character of $L(\mf{l}_Y,\la)$. This proves (i).

The proof for (ii) is analogous and hence omitted.
\end{proof}

\begin{lem}\label{high:to:high}
If $\wt{M}$ is a highest weight $\DG$-module of highest weight
$\la^\theta$ with $\la\in\mc{P}_Y$, then $T(\wt{M})$ and
$\ov{T}(\wt{M})$ are highest weight $\G$- and $\SG$-modules of
highest weight $\la\in \mc{P}_Y$ and $\la^\natural\in{\ov{\mc{P}}_Y}$,
respectively.
\end{lem}

\begin{proof}
We will only show this for $T(\wt{M})$, as the case of $\ov{T}(\wt{M})$ is analogous. Let
$v$ be a nonzero vector in $\wt{M}$ of weight $\la$ obtained from a non-zero vector of
weight $\la^\theta$ by applying the sequence of odd reflections of \secref{afterkk}. Such
a vector by \corref{aux:cor1} is a $\wt{\mf{b}}^c(n)$-highest weight vector of the
$\DG$-module $\wt{M}$, for $n\gg 0$. Evidently $v\in T(\wt{M})$ and, since
$\mf{b}=\wt{\mf{b}}^c(n)\cap\G$, $v$ is a $\mf{b}$-singular vector. The $\G$-module
$T(\wt{M})$, regarded as an $\mf{l}_Y$-module, is completely reducible by
\lemref{lem:aux1}. Thus to prove the lemma it is enough to show that every vector $w\in
T(\wt{M})$ of weight $\mu\in \mc{P}_Y$ lies in $\mc{U}(\mf{n}_-)v$. To see this, choose
$n$ so that $\ell(\la_+)<n$ and $\ell(\mu_+)<n$. Then with respect to $\wt{\mf{b}}^c(n)$,
$v$ is a highest weight vector of $\wt{M}$ and hence $w\in
\mc{U}(\wt{\mf{n}}^c_{-}(n))v$, where $\wt{\mf{n}}^c_{-}(n)$ is the opposite nilradical
of $\wt{\mf{b}}^c(n)$.  Now the conditions $\ell(\la_+)<n$ and $\ell(\mu_+)<n$ imply that
\begin{equation}\label{la-mu}
\la-\mu=\sum_{i=-m}^{-1}a_i\epsilon_i+\sum_{j=1}^{n-1}b_j\epsilon_j,\quad a_i,b_j\in\Z
\end{equation}
But $\la-\mu$ is also a finite $\Z_+$-linear combination of simple roots from
$\wt{\Pi}^c(n)$.  So we can write
\begin{equation*}
\la-\mu=\sum_{\alpha\in\wt{\Pi}^c(n)}a_\alpha\alpha,\quad a_\alpha\in\Z_+.
\end{equation*}
If there were some $\alpha\in\{\epsilon_{n}-\epsilon_{1/2};
\beta_{1/2},\cdots,
\beta_{n-1/2};\alpha_{n+1/2},\alpha_{n+1},\cdots\}$ with
$a_\alpha\not=0$, then it is easy to see that
$\langle\la-\mu,E_{rr}\rangle\not=0$, for
$r\not\in[-m,-1]\cup[1,n-1]$. It contradicts \eqnref{la-mu}.
Therefore $\la-\mu$ is a $\Z_+$-linear combination of
$\{\alpha_{-m},\cdots,\alpha_{-2};\beta_{-1};\beta_1,\cdots,\beta_{n-1}\}$,
and hence $w\in \mc{U}(\mf{n}_-)v$.
\end{proof}

\begin{thm}\label{matching:modules} For $\la\in\mc{P}_Y$, we have
\begin{align*}
&T(\wt{K}(\la^\theta))=K(\la),\quad T(\wt{L}(\la^\theta))=L(\la);\\
&\ov{T}(\wt{K}(\la^\theta))=\ov{K}(\la^\natural),\quad \ov{T}(\wt{L}(\la^\theta))=\ov{L}(\la^\natural).
\end{align*}
\end{thm}

\begin{proof} We will show this for $T$. The argument for $\ov{T}$ is analogous.  Computing the character of $\wt{K}(\la^\theta)$ we have (see~\eqnref{eqn:hookschur})
\begin{equation*}
{\rm ch}{\wt{K}}(\la^\theta)=
\prod_{i<0,j\in\N,r\in\hf+\Z_+}\frac{(1+x_{i}^{-1}x_r)}{(1-x_{i}^{-1}x_j)}
{\rm ch}\big{(}{\rm
Ind}_{\wt{\mf{p}}_Y^{<0}}^{\DG^{<0}}{L(\wt{\mf{l}}^{<0}_Y,\la^{<0}})\big{)}HS_{\la_+'}(x_{1\over
2},x_1,x_{3 \over 2},x_2,\cdots).
\end{equation*}
Application of $T$ amounts to setting the variables $x_r=0$, $r\in\hf+\N$.  Thus
\begin{equation*}
{\rm
ch}T(\wt{K}(\la^\theta))=\prod_{i<0,j\in\N}\frac{1}{(1-x_{i}^{-1}x_j)}{\rm
ch}\big{(}{\rm
Ind}_{\wt{\mf{p}}_Y^{<0}}^{\DG^{<0}}{L(\wt{\mf{l}}^{<0}_Y,\la^{<0}})\big{)}s_{\la_+}(x_1,x_2,\cdots),
\end{equation*}
which equals ${\rm ch}K(\la)$.  Since $T(\wt{K}(\la^\theta))$ is
a highest weight module by \lemref{high:to:high}, we see that
$T(\wt{K}(\la^\theta))=K(\la)$.

Let $\wt{M}:=\wt{L}(\la^\theta)$ with $\la\in \mc{P}_Y$. Suppose
that $M:=T(\wt{M})$ is not irreducible. Since by
\lemref{high:to:high} the $\G$-module $M$ is a highest weight
module, it must have a $\mf{b}$-singular vector inside $M$ that is
not a highest weight vector.  Suppose that $w$ is such a $\mf{b}$-singular
vector of weight $\mu\in \mc{P}_Y$.  We can choose $n\gg 0$ such
that $\la$ is the highest weight of $\wt{M}$ with respect to
$\wt{\mf{b}}^c(n)$, and $\ell(\la_+)<n$ and $\ell(\mu_+)<n$. By
\corref{aux:cor1} there exists a $\wt{\mf{b}}^c(n)$-highest
weight vector $v_\la$ of the $\DG$-module $\wt{M}$ of weight $\la$. It is
clear that $v_\la$ is a $\mf{b}$-highest weight vector of $M$ and
hence $w\in \mc{U}(\mf{a})v_\la$, where $\mf{a}$ is the subalgebra of
$\mf{n}_-$ generated by root vectors in $\mf{n}_-$ corresponding
to the roots $-\alpha_{-m},\cdots,-\alpha_{-2},-\beta_{-1}$ and
$-\beta_j$, $1\le j\le k$, for some $k$.

Choose $q\in \N$ such that $q\ge n$ and $q> k+1$. Note that
$v_\la$ is also a $\wt{\mf{b}}^c(q)$-highest weight vector of the
$\DG$-module $\wt{M}$ of weight $\la$. Since $w$ is
$\mf{b}$-singular it is annihilated by the root vectors
corresponding to the root
$\alpha_{-m},\cdots,\alpha_{-2},\beta_{-1}$ and $\beta_j$, for all
$j\in\N$. Also $w$ is annihilated by the root vectors
corresponding to the root in
$\wt{\Pi}^c(q)\backslash\{\alpha_{-m},\cdots,\alpha_{-2},\beta_{-1},\beta_1,\beta_2,\cdots,\beta_
{q-1}\}$ since $w\in \mc{U}(\mf{a})v_\la$ and these root vectors
commute with $\mf{a}$. It follows that $w$, regarded as in $\wt{M}$, is
then a $\wt{\mf{b}}^c(q)$-singular vector, contradicting the
irreducibility of $\wt{M}$.
\end{proof}

\begin{prop}\label{functor} $T$ and $\ov{T}$ define exact functors from $\wt{\mc{O}}_Y$ to $\mc{O}_Y$ and from $\wt{\mc{O}}_Y$ to $\ov{\mc{O}}_Y$, respectively.  Furthermore, $T$ and $\ov{T}$ send $\wt{\mc{O}}_Y^f$ to $\mc{O}_Y^f$ and $\wt{\mc{O}}_Y^f$  to $\ov{\mc{O}}_Y^f$, respectively.
\end{prop}

\begin{proof} Exactness is clear from the definitions.

It remains to prove that if $\wt{M}\in\wt{\mc{O}}_Y$, then $T(\wt{M})\in\mc{O}_Y$ and $\ov{T}(\wt{M})\in\ov{\mc{O}}_Y$. We will only prove $T(\wt{M})\in\mc{O}_Y$, as the proof of $\ov{T}(\wt{M})\in\ov{\mc{O}}_Y$ is analogous.

Clearly ${\rm dim}T(\wt{M})_\gamma<\infty$, for all $\gamma\in\h^*$. Now if
$\wt{M}\cong\bigoplus_{\mu\in\mc{P}_Y}L(\wt{\mf{l}}_Y,\mu^\theta)^{m(\mu)}$, then by
\lemref{lem:aux1} $T(\wt{M})\cong\bigoplus_{\mu\in\mc{P}_Y}L({\mf{l}}_Y,\mu)^{m(\mu)}$.
(Here and below $m(\mu)$ stands for the multiplicity of $L(\wt{\mf{l}}_Y,\mu^\theta)$ in
$\wt{M}$.)

Finally by \thmref{matching:modules}
$T(\wt{L}(\la^\theta))=L(\la)$.  By exactness of $T$ it follows
that a downward filtration for $\wt{M}$ gives rise to a
corresponding downward filtration of $T(\wt{M})$ with a one-to-one
correspondence between the composition factors. Hence
$T(\wt{M})\in\mc{O}_Y$.

The second part of the proposition is clear.
\end{proof}

\subsection{Some consequences}

Let $M\in\mc{O}_Y$. We may regard $\text{ch}M$ as an element in $\Z[[x^{\pm
1}_{-m},\cdots,x^{\pm 1}_{-1}]]\otimes\La_\Z(x_1,x_2,\cdots)$, where
$\La_\Z(x_1,x_2,\cdots)$ denotes the space of (completed) symmetric functions in the
variables $x_1,x_2,\cdots$. Similarly, for $\ov{M}\in\ov{\mc{O}}_Y$ and
$\wt{M}\in\wt{\mc{O}}_Y$, $\text{ch}\ov{M}$  and $\text{ch}\wt{M}$ may be viewed as
elements in $\Z[[x^{\pm 1}_{-m},\cdots,x^{\pm
1}_{-1}]]\otimes\La_\Z(x_{1/2},x_{3/2},\cdots)$ and $\Z[[x^{\pm 1}_{-m},\cdots,x^{\pm
1}_{-1}]]\otimes\La_\Z(x_1,x_2,\cdots)\otimes\La_\Z(x_{1/2},x_{3/2},\cdots)$,
respectively.

Let $\ov{\omega}:\La_\Z(x_1,x_2,\cdots)\rightarrow\La_\Z(x_{1/2},x_{3/2},\cdots)$ be the ring homomorphism that sends the $n$th complete symmetric function in $x_1,x_2,\cdots$ to the $n$th elementary symmetric function in $x_{1/2},x_{3/2},\cdots$.  Let $\wt{\omega}:\La_\Z(x_1,x_2,\cdots)\rightarrow \La_\Z(x_1,x_2,\cdots)\otimes\La_\Z(x_{1/2},x_{3/2},\cdots)$ be the ring homomorphism defined by sending the $n$th complete symmetric function in $x_2,x_4,\cdots$ (respectively in $x_1,x_3,\cdots$) to the $n$th complete symmetric function in $x_1,x_2,\cdots$ (respectively to the $n$th elementary symmetric function in $x_{1/2},x_{3/2},\cdots$).

\begin{cor}\label{comb:char1} Let $\la\in\mc{P}_Y$.  We have
\begin{itemize}
\item[(i)] $\wt{\omega}({\rm ch}L(\la))={\rm ch}\wt{L}(\la^\theta)$.
\item[(ii)] $\ov{\omega}({\rm ch}L(\la))={\rm ch}\ov{L}(\la^\natural)$.
\end{itemize}
\end{cor}

\begin{proof}
Since $\wt{\omega}\big{(}\text{ch}L({\mf{l}}_Y,\la)\big{)}=
\text{ch}L(\wt{\mf{l}}_Y,\la^\theta)$ and
$\ov{\omega}\big{(}\text{ch}L({\mf{l}}_Y,\la)\big{)}=\text{ch}L(\ov{\mf{l}}_Y,\la^\natural)$
the corollary follows directly from \lemref{lem:aux1} and \thmref{matching:modules}.
\end{proof}

\begin{rem} \corref{comb:char1} (ii) is consistent with the prediction of the super duality conjecture, and in the case of irreducible polynomial representations (i.e.~ $\la\in\mc{P}_{[-m,-2]}$ with $\la_{-1}\ge \la_1$) gives \cite[Theorem 6.10]{BR}. In the case of $Y=[-m,-2]$ it gives \cite[Corollary 6.15]{CWZ}. For infinite-dimensional unitary modules appearing in certain Howe dualities it also recovers \cite[Theorem 5.3]{CLZ}.
\end{rem}

For $n\in\N$, we recall the truncation functor
$\mf{tr}_{n}:\ov{\mc{O}}_Y\rightarrow(\ov{\mc{O}}_Y)_n$ of \cite[Definition 3.1]{CW2},
where here and further we use a subscript $n$ to indicate a corresponding truncated
category of $\gl(m|n)$-modules. For
$\gamma\in\sum_{i=-m}^{-1}\Z\epsilon_i+\sum_{j=1}^n\Z\epsilon_{j-{1\over 2}}$ let
$\ov{K}_n(\gamma)$ and $\ov{L}_n(\gamma)$ be the parabolic Verma $\gl(m|n)$-module and
irreducible $\gl(m|n)$-module of highest weight $\gamma$ in the category
$(\ov{\mc{O}}_Y)_n$, respectively. We recall the following.

\begin{lem}\label{cor:cw33} \cite[Corollary
3.3]{CW2} Let $\la\in{\ov{\mc{P}}_Y}$. The truncation functor $\mf{tr}_{n}$, for every
$n\in \N$, is exact and it sends $\ov{K}(\la)$ and $\ov{L}(\la)$ to $\ov{K}_n(\la)$ and
$\ov{L}_n(\la)$, respectively, if $\langle\la,E_{n+1/2,n+1/2}\rangle=0$, and to zero
otherwise.
\end{lem}

\begin{prop}\label{finite:filt}
The module $\ov{K}(\la)$ lies in
$\ov{\mc{O}}_Y^f$, for all $\la\in{\ov{\mc{P}}_Y}$.
Thus category $\ov{\mc{O}}_Y^f$ is the category of finitely
generated $\SG$-modules that as $\ov{\mf{l}}_Y$-modules are direct
sums of $L(\ov{\mf{l}}_Y,\mu)$, $\mu\in{\ov{\mc{P}}_Y}$, with a locally
nilpotent $\ov{\mf{u}}_Y$-action.
\end{prop}

\begin{proof}
Consider a fixed $\la\in{\ov{\mc{P}}_Y}$. Choose $n\gg 0$ so that $\langle\la,E_{n+{1/
2},n+{1/ 2}}\rangle=0$ and the degree of atypicality for $\la$ does not increase anymore
with increasing $n$.  Assume $\ov{L}(\mu)$ is a composition factor in $\ov{K}(\la)$. We
have $\mu\in{\ov{\mc{P}}_Y}$. Choose $k\ge n$ such that $\mf{tr}_{k}(\ov{L}(\mu))\not=0$.
Then $\la$ and $\mu$ share the same central character in $(\ov{\mc{O}}_Y)_k$. Therefore
our choice of $n$ together with $\mu\in{\ov{\mc{P}}_Y}$ implies that
$\langle\mu,E_{n+{1/2},n+{1/2}}\rangle=0$.  Thus by \lemref{cor:cw33} the multiplicity of
each $\ov{L}_n(\mu)$ inside each $\ov{K}_n(\la)$ is the same as that of $\ov{L}(\mu)$ in
$\ov{K}(\la)$. Since the $\gl(m|n)$-module $\ov{K}_n(\la)$ has finite composition series
(because as a $\gl(m|n)_{\bar{0}}$-module it is isomorphic to the tensor product of a
generalized Verma module and a finite-dimensional module), it follows that
$\ov{K}(\la)\in\ov{\mc{O}}_Y^f$.

By a standard argument a finitely generated $\SG$-module $\ov{M}$
that as $\ov{\mf{l}}_Y$-module is a direct sum of
$L(\ov{\mf{l}}_Y,\mu)$, $\mu\in{\ov{\mc{P}}_Y}$, with a locally
nilpotent $\ov{\mf{u}}_Y$-action, has a finite filtration by
highest weight modules, which are quotients of $\ov{K}(\la)$, for
$\la\in{\ov{\mc{P}}_Y}$. Thus $\ov{M}\in \ov{\mc{O}}_Y^f$ and
hence the proposition follows.
\end{proof}

\thmref{matching:modules} and \propref{finite:filt} give the
following.

\begin{cor}\label{G-finite:filt}
\begin{itemize}
\item[(i)] The module $\wt{K}(\la)\in \wt{\mc{O}}_Y^f$, for all
    $\la\in{\wt{\mc{P}}_Y}$.  Hence the category $\wt{\mc{O}}_Y^f$ is the category of
    finitely generated $\DG$-modules that as $\wt{\mf{l}}_Y$-modules are direct sums
    of $L(\wt{\mf{l}}_Y,\mu)$, $\mu\in{\wt{\mc{P}}_Y}$, with a locally nilpotent
    $\wt{\mf{u}}_Y$-action.
\item[(ii)] The module ${K}(\la)\in
{\mc{O}}_Y^f$, for all $\la\in\mc{P}_Y$. Hence the category ${\mc{O}}_Y^f$ is the category of
finitely generated $\G$-modules that as ${\mf{l}}_Y$-modules are
direct sums of $L({\mf{l}}_Y,\mu)$, $\mu\in\mc{P}_Y$, with a locally
nilpotent ${\mf{u}}_Y$-action.
\end{itemize}
\end{cor}

\begin{rem}\label{rmk:intro}
\propref{finite:filt} and its \corref{G-finite:filt} imply that the categories
$\ov{\mc{O}}^f_Y$ and ${\mc{O}}^f_Y$ are the categories $\mc{O}^{++}_{{\bf m}|\infty}$
and $\mc{O}^{++}_{{\bf m}+\infty}$ of \cite{CW2}, respectively.  We note that the proof
of \propref{finite:filt} that we have presented above is elementary. In the proof above
we have only used the rather easy \lemref{cor:cw33}.
\end{rem}

By \thmref{matching:modules} and \propref{functor} and \corref{G-finite:filt}, we have the following.

\begin{cor}\label{Composition-factors} Let $\la,\mu\in\mc{P}_Y$.
The numbers of composition factors of $\wt{K}(\la^\theta)$, ${K}(\la)$ and
$\ov{K}(\la^\natural)$ that are isomorphic to $\wt{L}(\mu^\theta)$, ${L}(\mu)$ and
$\ov{L}(\mu^\natural)$, respectively, are the same.
\end{cor}

Recall the super Bruhat ordering for weights in $\ov{\mf{h}}^*$ (see e.g.~\cite[\S
2-b]{B} or \cite[\S 2.3]{CW2}) which we denote by $\succcurlyeq$. Let us denote by $\ge$
the classical Bruhat ordering on $\mf{h}^*$. As a further application we present a super
analogue of a classical theorem of BGG.

\begin{cor} Let $\la\in{\mc{P}}_Y$ and $\gamma\in\ov{\mf{h}}^*$. If $\ov{L}(\gamma)$ is a subquotient of $\ov{K}(\la^\natural)$, then
$\la^\natural\succcurlyeq\gamma$.
\end{cor}

\begin{proof}
Clearly $\gamma=\mu^\natural$ for some $\mu\in\mc{P}_Y$. By \corref{Composition-factors}
$\ov{L}(\mu^\natural)$ is a subquotient of $\ov{K}(\la^\natural)$ if and only if
${L}(\mu)$ is a subquotient of ${K}(\la)$. By the classical version of the BGG Theorem
(e.g.~\cite[Section 5.1]{H}) $\mu\leq\la$.  Now \cite[Lemma 4.6]{CW2} implies that
$\mu^\natural\preccurlyeq\la^\natural$.
\end{proof}

\subsection{Irreducible characters}

\begin{thm}\label{character} Let $\la\in\mc{P}_Y$. Let ${\rm ch}L(\la)=\sum_{\mu\in \mc{P}_Y}a_{\mu\la}{\rm ch}K(\mu)$.  Then
\begin{itemize}
\item[(i)] ${\rm ch}\wt{L}(\la^\theta)=\sum_{\mu\in
\mc{P}_Y}a_{\mu\la}{\rm ch}\wt{K}(\mu^\theta)$,
\item[(ii)]
${\rm ch}\ov{L}(\la^\natural)=\sum_{\mu\in
\mc{P}_Y}a_{\mu\la}{\rm ch}\ov{K}(\mu^\natural)$.
\end{itemize}
\end{thm}

\begin{proof}
Since $\ov{\omega}(\text{ch}K(\la))=\text{ch}\ov{K}(\la^\natural)$ and
$\wt{\omega}(\text{ch}K(\la))=\text{ch}\wt{K}(\la^\theta)$, for $\la\in\mc{P}_Y$, the
theorem follows directly from \corref{comb:char1}.
\end{proof}

\begin{rem}\label{aux:339}
By \eqnref{aux:449} the coefficients $a_{\mu\la}$ in \thmref{character} equal
$\mf{l}_{\mu\la}(1)$, where $\mf{l}_{\mu\la}(q)$ are the classical (parabolic)
Kazhdan-Lusztig polynomials \cite{D, KL} (see also \cite[Proposition 4.4]{CW2}). Since by
\cite[Theorem 4.7]{CW2} the polynomials $\mf{l}_{\mu\la}(q)$ equal
$\ell_{\mu^\natural\la^\natural}(q)$ (see \eqnref{aux224}) \thmref{character} verifies
\cite[Conjecture 3.10]{CW2}, which is a parabolic version of a conjecture of Brundan
\cite[Conjecture 4.32]{B}. In particular, \thmref{character} in the special case
$Y=[-m,-2]$, together with \lemref{cor:cw33}, gives an independent new proof of the first
part of \cite[Theorem 4.37]{B}. We note that our results do not rely on \cite{B, Se}.
\end{rem}

Below we work out in more detail a character formula for the irreducible
$\gl(m|n)$-module $\ov{L}_n(\gamma)$, where $\gamma$ is a weight of the form
\begin{equation*}
\sum_{i=-m}^{-1}\gamma_i\epsilon_i+\sum_{j=1}^n\gamma_j\epsilon_{j-1/2},\quad\gamma_i,\gamma_j\in\Z,
\end{equation*}
with $\gamma_1\ge\gamma_2\ge\cdots\ge\gamma_n$. Recall that the one-dimensional
determinant module $\text{det}$ has (highest) weight ${\bf
1}_{m|n}=\sum_{i=-m}^{-1}\epsilon_i-\sum_{j=1}^n\epsilon_{j-1/2}$.  For $k\in\Z$ and an
$\ov{\mf{h}}$-semisimple $\gl(m|n)$-module $M$ with $\text{ch}M=\sum_{\eta}
\text{dim}{M}_\eta e^\eta$ we have
\begin{equation*}
\text{ch}(M\otimes\text{det}^{\otimes k})=\sum_{\eta}\text{dim}M_{\eta}e^{\eta+k{\bf 1}_{m|n}}.
\end{equation*}
Clearly $\ov{L}_n(\gamma+k{\bf 1}_{m|n})=\ov{L}_n(\gamma)\otimes\text{det}^{\otimes k}$.
Thus taking tensor product with a suitable power of the determinant module, if necessary,
we may assume that $\gamma_1\ge\gamma_2\ge\cdots\ge\gamma_n\ge 0$ and so
$\gamma\in\ov{\mc{P}}_Y$.

Let $\la\in\mc{P}_Y$ (with $Y=\emptyset$) be such that $\la^\natural=\gamma$.
We have $\ell(\la'_+)\le n$. Now
\thmref{character} (ii) (together with \lemref{cor:cw33})
implies that the character of the irreducible $\gl(m|n)$-module of
highest weight $\la^\natural$ equals to
\begin{align}\label{finite:char}
\text{ch}\ov{L}_n(\gamma)=
\sum_{\mu\in\mc{P}_Y,\ell(\mu'_+)\le n}a_{\mu\la}\text{ch}\ov{K}_n(\mu^\natural),
\end{align}
where $\ov{K}_n(\mu^\natural)$ is the parabolic Verma $\gl(m|n)$-module corresponding to
$Y=\emptyset$. As the coefficients $a_{\mu\la}$ are known by \remref{aux:339},
\eqnref{finite:char} gives the irreducible character for $\gl(m|n)$-module
$\ov{L}_n(\gamma)$. In the special case of
$\gamma_{-m}\ge\gamma_{-m+1}\ge\cdots\ge\gamma_{-1}$ we obtain an irreducible character
formula for finite-dimensional irreducible $\gl(m|n)$-module.  A formula (corresponding
to our case $Y=[-m,-2]$) was obtained in \cite{Se, B}.  \thmref{character} is obtained
using an approach very different from \cite{Se} and \cite{B}, and provides an independent
solution of the irreducible character problem.

\section{Kazhdan-Lusztig Polynomials}\label{section4}

\subsection{Homology of Lie superalgebras}\label{aux:homology} Let $L=L_{\bar{0}}\oplus L_{\bar{1}}$ be a
Lie superalgebra and let $\mc{T}(L)$ be the tensor algebra of $L$. Then
$\mc{T}(L)=\bigoplus_{n=0}^{\infty}\mc{T}^n(L)$ is an associative superalgebra with a
canonical $\Z$-gradation. For $v\in L_\epsilon$, we let $|v|:=\epsilon$,
$\epsilon\in\Z_2$. The exterior algebra of $L$ is the quotient algebra $\Lambda
(L):=\mc{T}(L)/J$, where $J$ is the homogeneous two-sided ideal of $\mc{T}(L)$ generated
by the elements of the form
  $$ x\otimes y+(-1)^{|x||y|}y\otimes x,$$
where $x$ and $y$ are homogeneous elements of $L$. The $\Lambda (L)$ is also an associative
superalgebra with a $\Z$-gradation inherited from $\mc{T}(L)$. More
precisely, we have $\Lambda (L)=\bigoplus_{n=0}^{\infty} \Lambda^n
L$, where $\Lambda^n L$ is the set of all homogeneous elements of
$\Z$-degree $n$ in $\Lambda (L)$, for each $n\ge 0$. For $\Z_2$-homogeneous elements
$x_1,x_2, \cdots,x_k\in L$, the image of the element $x_1\otimes
x_2\otimes\cdots\otimes x_k$ under the canonical quotient map from
$\mc{T}^k(L)$ to $\La^k(L)$ will be denoted by $x_1x_2\cdots x_k$.

For an $L$-module $V$, the $k$th homology group $H_k(L;V)$ of
$L$ with coefficient in $V$ is defined to be the $k$th homology
group of the following complex (see e.g.~\cite{T}):
\begin{equation*}
\cdots \stackrel{\partial}{\longrightarrow}\La^n(L)\otimes
V\stackrel{\partial}{\longrightarrow} \La^{n-1}(L)\otimes
V\stackrel{\partial}{\longrightarrow} \cdots
\stackrel{\partial}{\longrightarrow}\La^{1}(L)\otimes V
\stackrel{\partial}{\longrightarrow}\La^{0}(L)\otimes
V{\longrightarrow} 0,
\end{equation*}
where the boundary operator $\partial$ is given by
\begin{align}\label{aux:d}
&\partial(x_1  x_2\cdots x_n\otimes v)\\
:=&\sum_{1\le s<t\le
n}(-1)^{s+t+|x_s|\sum_{i=1}^{s-1}|x_i|+|x_t|\sum_{j=1}^{t-1}|x_j|+|x_s||x_t|}
[x_s,x_t]x_1\cdots\widehat{x}_s\cdots\widehat{x}_t\cdots x_n\otimes v\nonumber\\
&+\sum_{s=1}^n(-1)^{s+|x_s|\sum_{i=s+1}^{n}|x_i|}x_1\cdots\widehat{x}_s\cdots
x_n\otimes  x_sv.\nonumber
\end{align}
Here the $x_i$s are homogeneous elements in $L$ and $v\in V$.
Furthermore $[x_s,x_t]\in \La(L)$ denotes the linear term
corresponding to $[x_s,x_t]\in L$ and, as usual, $\widehat{y}$
indicates that the term $y$ is omitted.

\subsection{Comparison of homology groups}
Here and further we shall suppress the subscript $Y$ and denote $(\wt{\mf{u}}_Y)_-$,
$({\mf{u}}_Y)_-$ and $(\ov{\mf{u}}_Y)_-$ by $\wt{\mf{u}}_-$, ${\mf{u}}_-$ and
$\ov{\mf{u}}_-$, respectively.

For $\wt{M}\in\wt{\mc{O}}_Y^f$ we denote by $M=T(\wt{M})\in\mc{O}_Y^f$ and
$\ov{M}=\ov{T}(\wt{M})\in\ov{\mc{O}}_Y^f$. Let $\wt{d}:\Lambda({\mf{\wt{u}}}_-)\otimes
{\wt{M}}\rightarrow \Lambda({\mf{\wt{u}}}_-)\otimes {\wt{M}}$,
$d:\Lambda({\mf{u}}_-)\otimes {M}\rightarrow \Lambda({\mf{u}}_-)\otimes {M}$ and
$\ov{d}:\Lambda(\ov{\mf{u}}_-)\otimes \ov{M}\rightarrow \Lambda(\ov{\mf{u}}_-)\otimes
\ov{M}$ be the boundary operator of the complex of $\wt{\mf{u}}_-$-homology with
coefficients in $\wt{M}$, the boundary operator of the complex of $\mf{u}_-$-homology
with coefficients in $M$ and the boundary operator of the complex of
$\ov{\mf{u}}_-$-homology with coefficients in $\ov{M}$, respectively. Note that $\wt{d}$,
${d}$, and $\ov{d}$ are $\wt{\mf{l}}_Y$-homomorphism, $\mf{l}_Y$-homomorphism and
$\ov{\mf{l}}_Y$-homomorphism, respectively.

 The following lemma is easy.
\begin{lem}\label{lem:aux2}
We have
\begin{itemize}
\item[(i)] $T\big{(}\La(\wt{\mf{u}}_-)\big{)}=\La({\mf{u}}_-)$,
\item[(ii)] $\ov{T}\big{(}\La(\wt{\mf{u}}_-)\big{)}=\La({\ov{\mf{u}}}_-)$.
\end{itemize}
\end{lem}

The $\mf{l}_Y$-module $\La(\mf{u}_-)$ is a direct sum of
$L(\mf{l}_Y,\mu)$, $\mu\in \mc{P}_Y$, each appearing with finite
multiplicity. Using \cite{S, BR} one can show that
$\La(\ov{\mf{u}}_-)$, as an $\ov{\mf{l}}_Y$-module, is a direct
sum of $L(\ov{\mf{l}}_Y,\mu^\natural)$, $\mu\in \mc{P}_Y$, each
appearing with finite multiplicity (\cite[Lemma 3.2]{CK}).
Similarly it follows that  that $\La(\wt{\mf{u}}_-)$, as an
$\wt{\mf{l}}_Y$-module, is a also direct sum of
$L(\wt{\mf{l}}_Y,\mu^\theta)$, $\mu\in \mc{P}_Y$, each appearing
with finite multiplicity (\cite[Section 3.2.3]{CK}).

The $\mf{l}_Y$-module $\La(\mf{u}_-)\otimes M$ is of course
completely reducible. The $\wt{\mf{l}}_Y$-module
$\Lambda(\wt{\mf{u}}_-)\otimes\wt{M}$ and $\ov{\mf{l}}_Y$-module
$\Lambda(\ov{\mf{u}}_-)\otimes\ov{M}$ are completely reducible by
\cite[Theorem 3.2]{CK} and \cite[Theorem 3.1]{CK}, respectively.

\begin{lem}\label{boundary}
For $\wt{M}\in\wt{\mc{O}}_Y^f$ and $\la\in\mc{P}_Y$, we have
\begin{itemize}\label{lem:aux3}
\item[(i)] $T\big{(}\La(\wt{\mf{u}}_-)\otimes\wt{M}\big{)}=\La({\mf{u}}_-)\otimes{M}$,
 and thus
 $T\big{(}\La(\wt{\mf{u}}_-)\otimes\wt{L}(\la^\theta)\big{)}=\La({\mf{u}}_-)\otimes{L}(\la)$.
 Moreover, $T[\wt{d}]=d$.
\item[(ii)]
$\ov{T}\big{(}\La(\wt{\mf{u}}_-)\otimes\wt{M}\big{)}=\La({\ov{\mf{u}}}_-)\otimes\ov{M}$,
and thus
$\ov{T}\big{(}\La(\wt{\mf{u}}_-)\otimes\wt{L}(\la^\theta)\big{)}
=\La({\ov{\mf{u}}}_-)\otimes\ov{L}(\la^\natural)$. Moreover,
$\ov{T}[\wt{d}]=\ov{d}$.
\end{itemize}
\end{lem}

\begin{proof} By \lemref{lem:aux2}, \thmref{matching:modules} and the compatibility of $T$ and $\ov{T}$
under tensor product we have the first part of (i) and (ii). Using
the definitions \eqnref{aux:d} of $\wt{d}$, $d$ and $\ov{d}$, we
have $\wt{d}(v)=d(v)$ for all $v\in \La({\mf{u}}_-)\otimes{M}$ and
$\wt{d}(w)=\ov{d}(w)$ for all $v\in
\La({\ov{\mf{u}}}_-)\otimes\ov{M}$. Hence we have $T[\wt{d}]=d$
and $\ov{T}[\wt{d}]=\ov{d}$.
\end{proof}

Lemmas \ref{lem:aux1} and \ref{lem:aux3} now imply the following.

\begin{lem}\label{lem:aux4} Suppose $\Lambda(\wt{\mf{u}}_-)\otimes\wt{M}\cong\bigoplus_{\mu\in \mc{P}_Y}L(\wt{\mf{l}}_Y,\mu^\theta)^{m(\mu)}$, as $\wt{\mf{l}}_Y$-modules.  Then
\begin{itemize}
\item[(i)] $\Lambda({\mf{u}}_-)\otimes{M}\cong\bigoplus_{\mu\in
\mc{P}_Y}L({\mf{l}}_Y,\mu)^{m(\mu)}$, as ${\mf{l}}_Y$-modules.
\item[(ii)]
$\Lambda(\ov{\mf{u}}_-)\otimes\ov{M}\cong\bigoplus_{\mu\in
\mc{P}_Y}L(\ov{\mf{l}}_Y,\mu^\natural)^{m(\mu)}$, as
$\ov{\mf{l}}_Y$-modules.
\end{itemize}
\end{lem}

By \lemref{boundary} and \eqnref{T}, we have the following commutative diagram.
\begin{eqnarray}\label{compare-complexes}
\CD
\cdots @>>> \La^{n+1}(\wt{\mf{u}}_-)\otimes\wt{M} @>\wt{d}>>\La^{n}(\wt{\mf{u}}_-)\otimes\wt{M} @>\wt{d}>>\La^{n-1}(\wt{\mf{u}}_-)\otimes\wt{M} @>\wt{d}>>\cdots \\
@.  @VVT_{\La^{n+1}(\wt{\mf{u}}_-)\otimes\wt{M}}V @VVT_{\La^{n}(\wt{\mf{u}}_-)\otimes\wt{M}}V @VVT_{\La^{n-1}(\wt{\mf{u}}_-)\otimes\wt{M}}V\\
\cdots @>>> \La^{n+1}({\mf{u}}_-)\otimes {M} @>{d}>>\La^{n}({\mf{u}}_-)\otimes{M} @>{d}>>\La^{n-1}({\mf{u}}_-)\otimes{M} @>{d}>>\cdots\\
 \endCD
\end{eqnarray}
Thus $T$ induces an $\mf{l}_Y$-homomorphism from $H_n(\wt{\mf{u}}_-;\wt{M})$ to $H_n({\mf{u}}_-;{M})$.
Similarly, $\ov{T}$ induces an $\ov{\mf{l}}_Y$-homomorphism from $H_n(\wt{\mf{u}}_-;\wt{M})$ to $H_n(\ov{\mf{u}}_-;\ov{M})$. Moreover, we have the following.

\begin{thm}\label{matching:KL} We have for $n\ge 0$
\begin{itemize}
\item[(i)] $T(H_n(\wt{\mf{u}}_-;\wt{M}))\cong H_n({\mf{u}}_-;{M})$, as $\mf{l}_Y$-modules.
\item[(ii)] $\ov{T}(H_n(\wt{\mf{u}}_-;\wt{M}))\cong H_n(\ov{\mf{u}}_-;\ov{M})$, as $\ov{\mf{l}}_Y$-modules.
\end{itemize}
\end{thm}

\begin{proof} We shall only prove (i), as the argument for (ii) is
parallel. By \lemref{boundary} and \eqnref{compare-complexes}, we have
 $$T({\rm Ker}(\wt{d}))={\rm Ker}(\wt{d}) \cap (\La({\mf{u}}_-)\otimes{M})={\rm Ker}(d)$$ and
 $$ T({\rm Im}(\wt{d}))={\rm Im}(\wt{d}) \cap (\La({\mf{u}}_-)\otimes{M})={\rm Im}(d).$$
Since $T$ is an exact functor, we have
 $$T(\bigoplus_{n\ge 0}H_n(\wt{\mf{u}}_-;\wt{M}))= T({\rm Ker}(\wt{d}))/T({\rm Im}(\wt{d}))= {\rm Ker}(d)/{\rm Im}(d)=\bigoplus_{n\ge 0}H_n({\mf{u}}_-;{M}).$$ This
completes the proof of the theorem.
%
\end{proof}

\thmref{matching:modules} implies the following.
\begin{cor}\label{matching:KL1} For $\la\in \mc{P}_Y$ and $n\ge
0$, we have
\begin{itemize}
\item[(i)] $T(H_n(\wt{\mf{u}}_-;\wt{L}(\la^\theta)))\cong
H_n({\mf{u}}_-;{L}(\la))$, as $\mf{l}_Y$-modules. \item[(ii)]
$\ov{T}(H_n(\wt{\mf{u}}_-;\wt{L}(\la^\theta)))\cong
H_n(\ov{\mf{u}}_-;\ov{L}(\la^\natural))$, as
$\ov{\mf{l}}_Y$-modules.
\end{itemize}
\end{cor}

\subsection{Kazhdan-Lusztig polynomials}\label{KL:polynomials}
Let $\gl_\infty$ be the infinite-dimensional general linear algebra with basis consisting of elementary matrices $E_{ij}$, $i,j\in\Z$. Let  $U_q(\gl_\infty)$ be its quantum group acting on the natural module $\mathbb V$ (see \cite[\S 2-c]{B} or \cite[\S 2.1]{CW2} for precise definition).  Let $\mathbb W$ be the restricted dual of $\mathbb V$ (\cite[\S 2-d]{B}, \cite[\S 2.3]{CW2}). Let $m_1,\cdots,m_s\in\N$ with $\sum_{i=1}^sm_i=m$ and $\wt{\mf{l}}_Y^{<0}\cong\bigoplus_{i=1}^s\gl(m_i)$.  Consider a certain topological completion $\widehat{\mc{E}}^{{\bf m}|\infty}$ of the Fock space (\cite[\S 2-d]{B}, \cite[\S 2.3]{CW2})
\begin{equation*}
\mc{E}^{{\bf m}|\infty}:=\La^{m_1}(\mathbb V)\otimes\La^{m_2}(\mathbb V)\otimes\cdots\La^{m_s}(\mathbb V)\otimes \La^{\infty}(\mathbb W).
\end{equation*}
By arguments essentially going back to \cite{KL} (c.f.~\cite[Theorem 2.17]{B}) $\widehat{\mc{E}}^{{\bf m}|\infty}$ has three sets of distinguished basis, namely the standard, canonical and dual canonical basis, parameterized by $\mc{P}_Y$, denoted respectively by $\{\ov{K}_{f_{\la^\natural}}|\la\in \mc{P}_Y\}$, $\{\ov{U}_{f_{\la^\natural}}|\la\in \mc{P}_Y\}$, and $\{\ov{L}_{f_{\la^\natural}}|\la\in \mc{P}_Y\}$. Furthermore one has
\begin{align}\label{aux224}
\ov{U}_{f_{\la^\natural}}=\sum_{\mu\in \mc{P}_Y} u_{\mu^\natural\la^\natural}(q)\ov{K}_{f_{\mu^\natural}},\quad
\ov{L}_{f_{\la^\natural}}=\sum_{\mu\in \mc{P}_Y} \ell_{\mu^\natural\la^\natural}(q)\ov{K}_{f_{\mu^\natural}},
\end{align}
where $u_{\mu^\natural\la^\natural}(q)\in\Z[q]$ and $\ell_{\mu^\natural\la^\natural}(q)\in\Z[q^{-1}]$ \cite[(2.3)]{CW2}, which are parabolic versions of \cite[(2.18)]{B}.

The following theorem is an analogue of Vogan's cohomological
interpretation of the Kazhdan-Lusztig polynomials.

\begin{thm}\label{cw2:conjecture} We have for $\la,\mu\in \mc{P}_Y$
\begin{equation*}
\ell_{\mu^\natural\la^\natural}(-q^{-1})=\sum_{n=0}^\infty{\rm
dim}_\C{\rm
Hom}_{\ov{\mf{l}}_Y}\big{(}L(\ov{\mf{l}}_Y,\mu^\natural),H_n(\ov{\mf{u}}_-;\ov{L}(\la^\natural))\big{)}q^n.
\end{equation*}
\end{thm}

\begin{proof} Consider a topological completion $\widehat{\mc{E}}^{{\bf m}+\infty}$ of the Fock space \cite[\S 2.2]{CW2}
\begin{equation*}
\mc{E}^{{\bf m}+\infty}:=\La^{m_1}(\mathbb V)\otimes\La^{m_2}(\mathbb V)\otimes\cdots\La^{m_s}(\mathbb V)\otimes \La^{\infty}(\mathbb V).
\end{equation*}
$\widehat{\mc{E}}^{{\bf m}+\infty}$ has the standard, canonical and dual canonical basis, parameterized by $\mc{P}_Y$, denoted respectively by $\{{K}_{f_{\la}}|\la\in \mc{P}_Y\}$, $\{{U}_{f_{\la}}|\la\in \mc{P}_Y\}$, and $\{{L}_{f_{\la}}|\la\in \mc{P}_Y\}$. Similarly to \eqnref{aux224} one has
\begin{align*}
{U}_{f_{\la}}=\sum_{\mu\in \mc{P}_Y} \mf{u}_{\mu\la}(q){K}_{f_{\mu}},\quad
{L}_{f_{\la}}=\sum_{\mu\in \mc{P}_Y} \mf{l}_{\mu\la}(q){K}_{f_{\mu}},
\end{align*}
where $\mf{u}_{\mu\la}(q)\in\Z[q]$ and $\mf{l}_{\mu\la}(q)\in\Z[q^{-1}]$.  It is folklore that (c.f.~\cite[Theorems 4.15 and 4.16]{CW2})
\begin{equation}\label{aux:449}
\text{ch}{L}(\la)=\sum_{\mu\in \mc{P}_Y}
\mf{l}_{\mu\la}(1)\text{ch}{K}({\mu}).
\end{equation}
From \cite[Conjecture 3.4]{V} and the Kazhdan-Lusztig
conjecture proved in \cite{BB, BK} we conclude
\begin{equation*}
\mf{l}_{\mu\la}(-q^{-1})=\sum_{n=0}^\infty{\rm dim}_\C{\rm
Hom}_{\mf{l}_Y}\big{(}L(\mf{l}_Y,\mu),H_n(\mf{u}_-;L(\la))\big{)}q^n.
\end{equation*}
Now by \cite[Theorem 4.7]{CW2} we have
\begin{equation*}
\ell_{\mu^\natural\la^\natural}(q)=\mf{l}_{\mu\la}(q).
\end{equation*}
By \corref{matching:KL1} we have
\begin{equation*}
{\rm dim}_\C{\rm Hom}_{\mf{l}_Y}\big{(}L(\mf{l}_Y,\mu),H_n(\mf{u}_-;L(\la))\big{)} =
{\rm dim}_\C{\rm Hom}_{\ov{\mf{l}}_Y}\big{(}L(\ov{\mf{l}}_Y,\mu^\natural),H_n(\ov{\mf{u}}_-;\ov{L}(\la^\natural))\big{)},
\end{equation*}
and hence the theorem follows.
\end{proof}

\begin{rem}
Let $\ov{U}(\la^\natural)$ denote the tilting module in $\ov{\mc{O}}^f_Y$ corresponding to $\la^\natural\in {\ov{\mc{P}}_Y}$ \cite[Theorem 3.14]{CW2}. By \remref{aux:339} we have
$\text{ch}\ov{L}(\la^\natural)=\sum_{\mu\in \mc{P}_Y} \ell_{\mu^\natural\la^\natural}(1)\text{ch}\ov{K}({\mu^\natural})$. This, together with the remark following \cite[Conjecture 3.10]{CW2}, implies that
\begin{equation*}
\text{ch}\ov{U}(\la^\natural)=\sum_{\mu\in\mc{P}_Y}u_{\mu^\natural\la^\natural}(1)\text{ch}\ov{K}(\mu^\natural).
\end{equation*}
\end{rem}

\section{Super Duality}\label{section5}

\subsection{Equivalence of the categories $\mc{O}_Y^f$ and $\ov{\mc{O}}^f_Y$}

The goal of this section is to establish the following.

\begin{thm}\label{thm:equivalence} Recall $T$ and $\ov{T}$ from \secref{Tfunctors}. We have the following.
\begin{itemize}
\item[(i)] $T:\wt{\mc{O}}_Y^f\rightarrow\mc{O}_Y^f$ is an equivalence of categories.
\item[(ii)] $\ov{T}:\wt{\mc{O}}_Y^f\rightarrow\ov{\mc{O}}_Y^f$ is an equivalence of categories.
\item[(iii)] The categories $\mc{O}_Y^f$ and $\ov{\mc{O}}_Y^f$ are equivalent.
\end{itemize}
\end{thm}

Since by \secref{sec:O} $\ov{\mc{O}}^{f, {\bar{0}}}_Y\equiv\ov{\mc{O}}_Y^f$ and $\wt{\mc{O}}^{f, {\bar{0}}}_Y\equiv\wt{\mc{O}}_Y^f$ it is enough to prove \thmref{thm:equivalence} for $\ov{\mc{O}}^{f, {\bar{0}}}_Y$ and $\wt{\mc{O}}^{f, {\bar{0}}}_Y$. In order to keep notation simple we will from now on drop the superscript $\bar{0}$ and use
$\ov{\mc{O}}_Y^f$ and $\wt{\mc{O}}_Y^f$ to denote the respective categories $\ov{\mc{O}}^{f, {\bar{0}}}_Y$ and $\wt{\mc{O}}^{f, {\bar{0}}}_Y$ for the remainder of the article.  Henceforth, when we write $\wt{K}(\la^\theta)\in\wt{\mc{O}}^f_Y$ and $\wt{L}(\la^\theta)\in\wt{\mc{O}}^f_Y$, $\la\in\mc{P}_Y$, we will mean the corresponding modules equipped with the $\Z_2$-gradation \eqnref{wt-Z2-gradation}.  Similar convention applies to $\ov{K}(\la^\natural)\in\ov{\mc{O}}^f_Y$ and $\ov{L}(\la^\natural)\in\ov{\mc{O}}^f_Y$.

For ${M},{N}\in{\mc{O}}_Y^f$ and $i\in \N$ the $i$th
extension ${\rm Ext}_{{\mc{O}}_Y^f}^i({M},{N})$ can be
understood in the sense of Baer-Yoneda (see e.g.~\cite[Chapter
VII]{M}) and ${\rm Ext}_{{\mc{O}}_Y^f}^0({M},{N}):={\rm
Hom}_{{\mc{O}}_Y^f}({M},{N})$. In a similar way
extensions in $\ov{\mc{O}}_Y^f$ and $\wt{\mc{O}}_Y^f$ can be
interpreted. From this view point the exact functors $T$ and $\ov{T}$
induce natural maps on extensions by taking the projection of the
corresponding exact sequences.

For $\wt{M}\in\wt{\mc{O}}_Y^f$ we let $M=T(\wt{M})$ and
$\ov{M}=\ov{T}(\wt{M})$. Since all the proofs in this section for
the functors $T$ and $\ov{T}$ are parallel, we shall only give
proofs for the functor $T$ without further explanation.

\begin{lem}\label{g-extension} Let
\begin{equation*}
 0\longrightarrow A\longrightarrow
B{\longrightarrow}C\longrightarrow 0
\end{equation*}
be an exact sequence of $\wt{\G}$-modules (respectively,
$\G$-modules, $\SG$-modules) such that $A$, $C\in \wt{\mc{O}}_Y^f$
(respectively, ${\mc{O}}_Y^f$, $\ov{\mc{O}}_Y^f$). Then $B$ also
belongs to $\wt{\mc{O}}_Y^f$ (respectively, ${\mc{O}}_Y^f$,
$\ov{\mc{O}}_Y^f$).
\end{lem}

\begin{proof}
The statement for $\mc{O}^f_Y$ is clear.  The statements for the categories $\ov{\mc{O}}^f_Y$ and $\wt{\mc{O}}^f_Y$ follow, for example, from \cite[Theorems 3.1 and 3.2]{CK}.
\end{proof}

For $\DG$-(respectively, $\SG$-, and $\G$-)modules  $A$ and $C$, let
${\mc Ext}^i_{(\mc{U}(\DG),\mc{U}(\wt{{\mf{l}}}_Y))}(C,A)$
(respectively,
${\mc Ext}^i_{(\mc{U}(\SG),\mc{U}(\ov{{\mf{l}}}_Y))}(C,A)$ and
${\mc Ext}^i_{(\mc{U}(\G),\mc{U}({{\mf{l}}_Y)})}(C,A)$) denote the
$i$th relative extension group of $A$ by $C$ (see e.g.~\cite[Appendix D]{Ku}).

Let $C$ be a $\wt{\mf{u}}_Y$-(respectively, $\ov{\mf{u}}_Y$-, and
${\mf{u}}_Y$-)modules. Let $H^i(\wt{\mf{u}}_Y;C)$ (respectively,
$H^i(\ov{\mf{u}}_Y;C)$ and $H^i({\mf{u}}_Y;C)$\,) denote the $i$th
$\wt{\mf{u}}_Y$-(respectively, $\ov{\mf{u}}_Y$- and
${\mf{u}}_Y$-)cohomology group with coefficients in $C$. Let
$\mc{H}^i(\wt{\mf{u}}_Y;C)$ (respectively,
$\mc{H}^i(\ov{\mf{u}}_Y;C)$ and $\mc{H}^i({\mf{u}}_Y;C)$\,) denote
the $i$th restricted (in the sense of \cite[Section 4]{L})
$\wt{\mf{u}}_Y$-(respectively, $\ov{\mf{u}}_Y$- and
${\mf{u}}_Y$-)cohomology group with coefficients in the $C$.

The following proposition is an analogue of \cite[\S7 Theorem 2]{RW}
(cf.~\cite[Lemma 9.1.8]{Ku}).

\begin{prop}\label{relative ext} Let $\la\in \mc{P}_Y$ and $\wt{N}\in \wt{\mc{O}}_Y$. For
$i\ge 0$, we have
\begin{itemize}
\item[(i)]\begin{align*}{\mc Ext}^i_{(\mc{U}(\wt{\G}),\mc{U}(\wt{{\mf{l}}}_Y))}(\wt{K}(\la^\theta),\wt{N})&\cong{\rm
Hom}_{\wt{\mf{l}}_Y}\big{(}
L(\wt{\mf{l}}_Y,\la^\theta),H^i(\wt{\mf{u}}_Y;\wt{N})\big{)}\\&\cong
{\rm Hom}_{\wt{\mf{l}}_Y}\big{(}
L(\wt{\mf{l}}_Y,\la^\theta),\mc{H}^i(\wt{\mf{u}}_Y;\wt{N})\big{)}.\end{align*}
\item[(ii)]\begin{align*}{\mc Ext}^i_{(\mc{U}(\SG),\mc{U}(\ov{{\mf{l}}}_Y))}(\ov{K}(\la^\natural),\ov{N})&\cong{\rm
Hom}_{\ov{\mf{l}}_Y}\big{(}
L(\ov{\mf{l}}_Y,\la^\natural),H^i(\ov{\mf{u}}_Y;\ov{N})\big{)}\\&\cong
{\rm Hom}_{\ov{\mf{l}}_Y}\big{(}
L(\ov{\mf{l}}_Y,\la^\natural),\mc{H}^i(\ov{\mf{u}}_Y;\ov{N})\big{)}.\end{align*}
\item[(iii)] \begin{align*}
{\mc Ext}^i_{(\mc{U}({\G}),\mc{U}({{\mf{l}}}_Y))}(K(\la),{N}) &\cong
{\rm Hom}_{{\mf{l}}_Y}\big{(}
L({\mf{l}}_Y,\la),H^i({\mf{u}}_Y;{N})\big{)}\\&\cong {\rm
Hom}_{{\mf{l}}_Y}\big{(}
L({\mf{l}}_Y,\la),\mc{H}^i({\mf{u}}_Y;{N})\big{)}.
\end{align*}
\end{itemize}\end{prop}

\begin{proof} We have the following relative version of Koszul resolution for the trivial module $\wt{\mf {p}}_Y$-modules (see e.g.~\cite[\S 1]{GL}):
\begin{equation}\label{Koszul}
\cdots\stackrel{\wt{\partial}_{k+1}}{\longrightarrow}\wt{C}_k\stackrel{\wt{\partial}_k}
{\longrightarrow}\wt{C}_{k-1}\stackrel{\wt{\partial}_{k-1}}{\longrightarrow}\cdots
\stackrel{\wt{\partial}_1}{\longrightarrow}\wt{C}_0\stackrel{\wt
{\epsilon}}{\longrightarrow} \C\longrightarrow 0,
\end{equation}
where $\wt{C}_k:=\mc{U}(\wt{\mf
{p}}_Y)\otimes_{\mc{U}(\wt{\mf{l}}_Y)}\Lambda^k(\wt{\mf{p}}_Y/\wt{\mf
{l}}_Y)$ is a $\wt{\mf{p}}_Y$-module with $\wt{\mf{p}}_Y$ acting on
the left of first factor for $k\ge 0$ and $\wt{\epsilon}$ is the
augmentation map from $\mc{U}(\wt{\mf{p}}_Y)$ to $\C$. The
$\wt{\mf{p}}_Y$-homomorphism $\wt{\partial}_k$ is given by
\begin{align}\label{Koszul-partial}
&\wt{\partial}_k(a\otimes \ov{x}_1  \ov{x}_2\cdots \ov{x}_k)\\
:=&\sum_{1\le s<t\le
k}(-1)^{s+t+|x_s|\sum_{i=1}^{s-1}|x_i|+|x_t|\sum_{j=1}^{t-1}|x_j|+|x_s||x_t|}
a\otimes\ov{[x_s,x_t]}\ov{x}_1\cdots\widehat{\ov{x}}_s\cdots\widehat{\ov{x}}_t\cdots \ov{x}_k\nonumber\\
&+\sum_{s=1}^k(-1)^{s+1+|x_s|\sum_{i=1}^{s-1}|x_i|}ax_s\otimes
\ov{x}_1\cdots\widehat{\ov{x}}_s\cdots \ov{x}_k.\nonumber
\end{align}
Here $a\in \mc{U}(\wt{\mf {p}}_Y)$ and the $x_i$s are homogeneous
elements in $\wt{\mf {p}}_Y$ and $\ov{x}_i$ denotes $x_i+\wt{\mf
{l}}_Y$ in $\wt{\mf{p}}_Y/\wt{\mf{l}}_Y$. As usual, $\widehat{y}$ indicates that the term $y$ is
omitted. Since $\wt{C}_k\cong
 \mc{U}(\wt{\mf
{u}}_Y)\otimes\Lambda^k\wt{\mf{u}}_Y$ as $\wt{\mf {l}}_Y$-module,
$\wt{C}_k$ is completely reducible as $\wt{\mf{l}}_Y$-module,
and hence the image of $\wt{\partial}_k$ is a direct summand of
$\wt{C}_{k-1}$.

Let $L(\wt{\mf{l}}_Y,\la^\theta)$ also denote the irreducible
$\wt{\mf{p}}_Y$-module on which $\wt{\mf{u}}_Y$ acts trivially. For $\la\in\mc{P}_Y$ and $k\geq 0$,
$\wt{D}_k:=\wt{C}_k\otimes L(\wt{\mf{l}}_Y,\la^\theta)$ is a
$\wt{\mf{p}}_Y$-module. Tensoring \eqnref{Koszul} with
$L(\wt{\mf{l}}_Y,\la^\theta)$ we obtain an exact sequence of
$\wt{\mf {p}}_Y$-modules
\begin{equation}\label{p:resolution}
\cdots\stackrel{\wt{d}_{k+1}}{\longrightarrow}\wt{D}_k\stackrel{\wt{d}_k}
{\longrightarrow}\wt{D}_{k-1}\stackrel{\wt{d}_{k-1}}{\longrightarrow}\cdots
\stackrel{\wt{d}_1}{\longrightarrow}\wt{D}_0\stackrel{\wt{d}_0}{\longrightarrow}
L(\wt{\mf{l}}_Y,\la^\theta)\longrightarrow 0,
\end{equation}
where $\wt{d}_k:=\wt{\partial}_k\otimes 1$ for $k>0$ and
$\wt{d}_0:=\wt{\epsilon}\otimes 1$.

 For $k\ge 0$, let
$\wt{E}_k:=\mc{U}(\DG)\otimes_{\mc{U}(\wt{\mf{p}}_Y)}\wt{D}_k$.
Tensoring \eqnref{p:resolution} with
$\mc{U}(\DG)\otimes_{\mc{U}(\wt{\mf{p}}_Y)}$ we obtain an exact
sequence of $\DG$-modules
\begin{equation}\label{projective-resolution}
\cdots\stackrel{\wt{\rho}_{k+1}}{\longrightarrow}\wt{E}_k\stackrel{\wt{\rho}_k}
{\longrightarrow}\wt{E}_{k-1}\stackrel{\wt{\rho}_{k-1}}{\longrightarrow}\cdots
\stackrel{\wt{\rho}_1}{\longrightarrow}\wt{E}_0\stackrel{\wt{\rho}_0}{\longrightarrow}
\wt{K}(\la^\theta)\longrightarrow 0,
\end{equation}
where $\wt{\rho}_k:=1\otimes\wt{d}_k$ for $k\ge 0$. We observe that
\begin{align*}
 \wt{E}_k&=\mc{U}(\DG)\otimes_{\mc{U}(\wt{\mf{p}}_Y)} \big{(}\wt{C}_k\otimes
L_0(\wt{\mf{l}}_Y,\la^\theta)\big{)}\\&=\mc{U}(\DG)\otimes_{\mc{U}(\wt{\mf{p}}_Y)}
\big{(}\mc{U}(\wt{\mf
{p}}_Y)\otimes_{\mc{U}(\wt{\mf{l}}_Y)}\big{(}\Lambda^k(\wt{\mf{p}}_Y/\wt{\mf
{l}}_Y)\otimes L(\wt{\mf{l}}_Y,\la^\theta)\big{)}\big{)}\\&\cong
\mc{U}(\DG)\otimes_{\mc{U}(\wt{\mf{l}}_Y)}\big{(}\Lambda^k(\wt{\mf{p}}_Y/\wt{\mf
{l}}_Y)\otimes L(\wt{\mf{l}}_Y,\la^\theta)\big{)}.
\end{align*}
By \cite[Lemma 3.1.7]{Ku} the $\wt{E}_k$s are
$(\mc{U}(\DG),\mc{U}(\wt{\mf{l}}_Y))$-projective modules. Since the image of
$\wt{\partial}_k$ in \eqnref{Koszul} is a $\mc{U}(\wt{\mf{l}}_Y)$-direct summand of
$\wt{C}_{k-1}$, the image of $\wt{\rho}_k$ in \eqnref{projective-resolution} is also a
$\mc{U}(\wt{\mf{l}}_Y)$-direct summand of $\wt{E}_{k-1}$, for $k\ge 1$, and hence
\eqnref{projective-resolution} is a $((\mc{U}(\DG),\mc{U}(\wt{\mf{l}}_Y))$-projective
resolution of $\wt{K}(\la^\theta)$. It follows therefore that the relative extension
group ${\mc
Ext}^i_{(\mc{U}(\wt{\G}),\mc{U}(\wt{{\mf{l}}}_Y))}(\wt{K}(\la^\theta),\wt{N})$ equals the
$i$th cohomology group of the following complex:
\begin{equation}\label{relative-Ext}
0\longrightarrow {\rm
Hom}_{\mc{U}(\wt{\G})}(\wt{E}_0,\wt{N})\stackrel{\wt{\rho}^*_1}
{\longrightarrow} {\rm
Hom}_{\mc{U}(\wt{\G})}(\wt{E}_1,\wt{N})\stackrel{\wt{\rho}^*_2}
{\longrightarrow} {\rm
Hom}_{\mc{U}(\wt{\G})}(\wt{E}_2,\wt{N})\stackrel{\wt{\rho}^*_3}
{\longrightarrow}\cdots
\end{equation}
Since
\[{\rm Hom}_{\mc{U}(\wt{\G})}(\wt{E}_i,\wt{N})\cong {\rm
Hom}_{\mc{U}(\wt{\mf{l}}_Y)}\big{(}L(\wt{\mf{l}}_Y,\la^\theta),{\rm
Hom}_{\C}(\Lambda^i\wt{\mf{u}}_Y,\wt{N})\big{)},\] for $i\ge 0$, the $i$th cohomology group of the \eqnref{relative-Ext} equals
${\rm Hom}_{\wt{\mf{l}}_Y}\big{(}
L(\wt{\mf{l}}_Y,\la^\theta),H^i(\wt{\mf{u}}_Y;\wt{N})\big{)}$ and
hence for each $i\ge 0$,
\begin{align*}{\mc Ext}^i_{(\mc{U}(\wt{\G}),\mc{U}(\wt{{\mf{l}}}_Y))}(\wt{K}(\la^\theta),\wt{N})&\cong{\rm
Hom}_{\wt{\mf{l}}_Y}\big{(}
L(\wt{\mf{l}}_Y,\la^\theta),H^i(\wt{\mf{u}}_Y;\wt{N})\big{)}.\end{align*}

Since the $\wt{\h}$-semisimple submodule of
$H^i(\wt{\mf{u}}_Y;\wt{N})$ with weights in $\wt{\Gamma}$
equals $\mc{H}^i(\wt{\mf{u}}_Y;\wt{N})$ we have
\begin{align*}
 {\rm Hom}_{\wt{\mf{l}}_Y}\big{(}
L(\wt{\mf{l}}_Y,\la^\theta),H^i(\wt{\mf{u}}_Y;\wt{N})\big{)} \cong
{\rm Hom}_{\wt{\mf{l}}_Y}\big{(}
L(\wt{\mf{l}}_Y,\la^\theta),\mc{H}^i(\wt{\mf{u}}_Y;\wt{N})\big{)}.
\end{align*}
This completes the proof of part (i). The proofs of (ii) and (iii)
are analogous.
\end{proof}

\begin{cor}\label{relative ext=}Let $\la\in \mc{P}_Y$ and $\wt{N}\in \wt{\mc{O}}_Y$. For
$i\ge 0$, we have
\[{\mc Ext}^i_{(\mc{U}(\wt{\G}),\mc{U}(\wt{{\mf{l}}}_Y))}(\wt{K}(\la^\theta),\wt{N})\cong
{\mc Ext}^i_{(\mc{U}(\SG),\mc{U}(\ov{{\mf{l}}}_Y))}(\ov{K}(\la^\natural),\ov{N})\cong
{\mc Ext}^i_{(\mc{U}({\G}),\mc{U}({{\mf{l}}}_Y))}(K(\la),{N}).\]
\end{cor}

\begin{proof}  Using the arguments of the proof of
\thmref{matching:KL}, we can show that
\[T\big{(}\mc{H}^i(\wt{\mf{u}}_Y;\wt{N})\big{)}\cong\mc{H}^i({\mf{u}}_Y;{N}),\]
and hence $T$ induces an isomorphism ($\forall\la\in \mc{P}_Y,
\forall i\in\Z_+$)
\begin{align}\label{T:cohomology}
T:{\rm Hom}_{\wt{\mf{l}}_Y}\big{(}
L(\wt{\mf{l}}_Y,\la^\theta),\mc{H}^i(\wt{\mf{u}}_Y;\wt{N})\big{)}
\stackrel{\cong}{\longrightarrow} {\rm Hom}_{{\mf{l}}_Y}\big{(}
L({\mf{l}}_Y,\la),\mc{H}^i({\mf{u}}_Y;{N})\big{)}.
\end{align}
By \propref{relative ext}, we have
\[{\mc Ext}^i_{(\mc{U}(\wt{\G}),\mc{U}(\wt{{\mf{l}}}_Y))}(\wt{K}(\la^\theta),\wt{N})\cong
{\mc Ext}^i_{(\mc{U}({\G}),\mc{U}({{\mf{l}}}_Y))}(K(\la),{N}).\]
Similarly, we have
\[{\mc Ext}^i_{(\mc{U}(\wt{\G}),\mc{U}(\wt{{\mf{l}}}_Y))}(\wt{K}(\la^\theta),\wt{N})\cong
{\mc Ext}^i_{(\mc{U}(\SG),\mc{U}(\ov{{\mf{l}}}_Y))}(\ov{K}(\la^\natural),\ov{N}).\]
\end{proof}

\begin{lem}\label{ext:uhomology}
Let $\la\in \mc{P}_Y$ and $\wt{N}\in \wt{\mc{O}}_Y$. For
$i=0,1$, we have
\begin{itemize}
\item[(i)] $T: {\rm
Ext}^i_{\wt{\mc{O}}_Y}(\wt{K}(\la^\theta),\wt{N}) \rightarrow {\rm
Ext}^i_{{\mc{O}}_Y}({K}(\la),{N})$ is an isomorphism,
\item[(ii)] $\ov{T}: {\rm
Ext}^i_{\wt{\mc{O}}_Y}(\wt{K}(\la^\theta),\wt{N})\rightarrow {\rm
Ext}^i_{{\ov{\mc{O}}_Y}}(\ov{K}(\la^\natural),\ov{N})$ is an
isomorphism.
\end{itemize}
\end{lem}

\begin{proof}
It is well known that
${\mc Ext}^1_{(\mc{U}(\wt{\G}),\mc{U}(\wt{{\mf{l}}}_Y))}(\wt{K}(\la^\theta),\wt{N})$
is isomorphic  to the equivalence classes of
$\wt{{\mf{l}}}_Y$-trivial extensions of $\wt{N}$ by
$\wt{K}(\la^\theta)$ \cite[\S2]{Ho} and
\[{\mc Ext}^0_{(\mc{U}(\wt{\G}),\mc{U}(\wt{{\mf{l}}}_Y))}(\wt{K}(\la^\theta),\wt{N})={\rm
Hom}_{\wt{\G}}(\wt{K}(\la^\theta),\wt{N}).\] Hence we have, for
$i=0,1$,
\[{\rm Ext}^i_{\wt{\mc{O}}_Y}(\wt{K}(\la^\theta),\wt{N})\cong
{\mc Ext}^i_{(\mc{U}(\wt{\G}),\mc{U}(\wt{{\mf{l}}}_Y))}(\wt{K}(\la^\theta),\wt{N}).\]
Similarly, for $i=0,1$, \[{\rm
Ext}^i_{{\mc{O}}_Y}(K(\la),{N})\cong{\mc Ext}^i_{(\mc{U}({\G}),\mc{U}({{\mf{l}}}_Y))}(K(\la),{N})\]
By \corref{relative ext=}, we have, for $i=0,1$,
\begin{equation}\label{Ext:isom}
 {\rm
Ext}^i_{\wt{\mc{O}}_Y}(\wt{K}(\la^\theta),\wt{N}) \cong {\rm
Ext}^i_{{\mc{O}}_Y}(K(\la),{N}).
\end{equation}
Since all the isomorphisms involved are natural, it is not hard
to see the isomorphism \eqnref{Ext:isom} is indeed induced by $T$.
This completes the proof of (i). Part (ii) is analogous.
\end{proof}

\begin{lem} Let $\wt{N}\in \wt{\mc{O}}_Y^f$ and
\begin{equation}\label{short:exact}
0\longrightarrow\wt{M}'\stackrel{\wt{i}}{\longrightarrow}\wt{M}{\longrightarrow}\wt{M}''\longrightarrow
0
\end{equation}
be an exact sequence of $\DG$-modules in $\wt{\mc{O}}_Y^f$. Then
\begin{itemize}
\item[(i)] The \eqnref{short:exact} induces the following
commutative diagram with exact rows. (We will use subscripts to distinguish various maps induced by $T$.)
\begin{eqnarray*}
\CD
 0  @>>> {\rm
Hom}_{\wt{\mc{O}}_Y^f}\big{(}\wt{M}'',\wt{N}\big{)} @>>>{\rm Hom}_{\wt{\mc{O}}_Y^f}\big{(}\wt{M},\wt{N}\big{)} @>>> {\rm Hom}_{\wt{\mc{O}}_Y^f}\big{(}\wt{M}',\wt{N}\big{)} @>\wt{\partial}>> \\
 @. @VVT_{\wt{M}'',\wt{N}}V @VVT_{\wt{M},\wt{N}}V @VVT_{\wt{M}',\wt{N}}V \\
 0  @>>> {\rm
Hom}_{{\mc{O}}_Y^f}\big{(}M'',{N}\big{)} @>>> {\rm
Hom}_{{\mc{O}}_Y^f}\big{(}M,{N}\big{)} @>>> {\rm
Hom}_{{\mc{O}}_Y^f}\big{(}M',{N}\big{)} @>\partial>>  \\
 \endCD
\end{eqnarray*}
\begin{eqnarray*}
\hskip 2cm\CD
@>\wt{\partial}>> {\rm Ext}^1_{\wt{\mc{O}}_Y^f}
\big{(}\wt{M}'',\wt{N}\big{)} @>>>{\rm Ext}^1_{\wt{\mc{O}}_Y^f}
\big{(}\wt{M},\wt{N}\big{)} @>>> {\rm Ext}^1_{\wt{\mc{O}}_Y^f}
\big{(}\wt{M}',\wt{N}\big{)} \\
 @. @VVT^1_{\wt{M}'',\wt{N}}V @VVT^1_{\wt{M},\wt{N}}V @VVT^1_{\wt{M}',\wt{N}}V \\
 @>\partial>> {\rm
Ext}^1_{{\mc{O}}_Y^f}\big{(}M'',{N}\big{)} @>>> {\rm
Ext}^1_{{\mc{O}}_Y^f}\big{(}M,{N}\big{)} @>>> {\rm
Ext}^1_{{\mc{O}}_Y^f}\big{(}M',{N}\big{)}\\
 \endCD
\end{eqnarray*}
\item[(ii)] The analogous statement holds replacing $T$ by
$\ov{T}$ in (i), $M$ by $\ov{M}$, et cetera.
\end{itemize}
\end{lem}

\begin{proof} We shall only prove (i), as the argument for (ii)
is analogous. By \cite[Chapter VII Proposition 2.2]{M}, the rows are exact. We
only need to show that the following diagram is commutative.
\begin{eqnarray*}
\CD
 {\rm
Hom}_{\wt{\mc{O}}_Y^f}\big{(}\wt{M}',\wt{N}\big{)}  @>\wt{\partial}>> {\rm Ext}^{1}_{\wt{\mc{O}}_Y^f}
\big{(}\wt{M}'',\wt{N}\big{)} \\
 @VVT_{\wt{M}',\wt{N}}V @VVT^1_{\wt{M}'',\wt{N}}V \\
 {\rm
Hom}_{{\mc{O}}_Y^f}\big{(}M',{N}\big{)}  @>\partial>> {\rm Ext}^{1}_{{\mc{O}}_Y^f}\big{(}M'',{N}\big{)}
 \endCD
\end{eqnarray*}
Let
 $\wt{f}\in {\rm
Hom}_{\wt{\mc{O}}_Y^f}\big{(}\wt{M}',\wt{N}\big{)}$. Then
$\wt{\partial}(\wt{f})\in {\rm Ext}^{1}_{\wt{\mc{O}}_Y^f}
\big{(}\wt{M}'',\wt{N}\big{)}$ is bottom exact row of the following commutative diagram:
\begin{eqnarray*}
\CD
 0  @>>> \wt{M'} @>\wt{i}>>\wt{M} @>>> \wt{M}'' @>>> 0 \\
 @. @VV\wt{f}V @VVV @|\\
 0  @>>> \wt{N} @>>>\wt{E} @>>> \wt{M}'' @>>> 0 \\
 \endCD
\end{eqnarray*}
Here $\wt{E}$ is the pushout of $\wt{f}$ and $\wt{i}$, and all
maps are the obvious ones. Let $f:=T_{\wt{M}',\wt{N}}[\wt{f}]$.
Since $T$ is compatible with pushouts,
$T^{1}_{\wt{M}'',\wt{N}}(\wt{\partial}(\wt{f}))=\partial(f)$,
which is the pushout of $f$ and $T[\wt{i}]$. On the other hand,
$T_{\wt{M}',\wt{N}}[\wt{f}]=f$ and hence
$\partial(T_{\wt{M}',\wt{N}}[\wt{f}])=\partial(f)$. This completes
the proof.
\end{proof}

By \lemref{lem:aux1} and the fact that
$T(\wt{\varphi})\not=0$ and $\ov{T}(\wt{\varphi})\not=0$ for any nonzero $\wt{\mf{l}}_Y$-homomorphism
$\wt{\varphi}$ from $L(\wt{\mf{l}}_Y,\la^\theta)$ to itself with $\la\in\mc{P}_Y$, we have the following.

\begin{lem}\label{Hom:MN}
Let $\wt{M},\wt{N}\in\wt{\mc{O}}_Y$. We have
\begin{itemize}
\item[(i)] $T:{\rm
Hom}_{\wt{\mc{O}}_Y^f}(\wt{M},\wt{N})\rightarrow {\rm
Hom}_{{\mc{O}}_Y^f}(M,N)$ is an injection,
 \item[(ii)]
$\ov{T}:{\rm Hom}_{\wt{\mc{O}}_Y^f}(\wt{M},\wt{N})\rightarrow {\rm
Hom}_{\ov{{\mc{O}}}_Y^f}(\ov{M},\ov{N})$ is an injection.
\end{itemize}
\end{lem}

\begin{lem}\label{lem:aux12}
Let $\la\in \mc{P}_Y$ and $\wt{N}\in\wt{\mc{O}}_Y^f$. We have
\begin{itemize}
\item[(i)] $T:{\rm
Hom}_{\wt{\mc{O}}_Y^f}(\wt{L}(\la^\theta),\wt{N}) \rightarrow {\rm
Hom}_{{\mc{O}}_Y^f}({L}(\la),N)$ is an isomorphism. \item[(ii)]
$\ov{T}:{\rm
Hom}_{\wt{\mc{O}}_Y^f}(\wt{L}(\la^\theta),\wt{N})\rightarrow {\rm
Hom}_{\ov{{\mc{O}}}_Y^f}(\ov{L}(\la^\natural),\ov{N})$ is an
isomorphism.
\end{itemize}
\end{lem}

\begin{proof} Consider the commutative diagram with exact rows
\begin{eqnarray*}
\CD
 0  @>>> \wt{M} @>>>\wt{K}(\la^\theta) @>>> \wt{L}(\la^\theta) @>>> 0 \\
 @. @VVT_{\wt{M}}V @VVT_{\wt{K}(\la^\theta)}V @VVT_{\wt{L}(\la^\theta)}V\\
 0  @>>> M @>>>{K}(\la) @>>> L(\la) @>>> 0 \\
 \endCD
\end{eqnarray*}
We obtain the following commutative diagram with exact rows.
\begin{eqnarray*}
\CD
 0  @>>> {\rm Hom}_{\wt{\mc{O}}_Y^f}\big{(}\wt{L}(\la^\theta),\wt{N}\big{)} @>>>{\rm Hom}_{\wt{\mc{O}}_Y^f}\big{(}\wt{K}(\la^\theta),\wt{N}\big{)} @>>> {\rm Hom}_{\wt{\mc{O}}_Y^f}\big{(}\wt{M},\wt{N}\big{)} \\
 @. @VVT_{\wt{L}(\la^\theta),\wt{N}}V @VVT_{\wt{K}(\la^\theta),\wt{N}}V @VVT_{\wt{M},\wt{N}}V\\
 0  @>>> {\rm
Hom}_{{\mc{O}}_Y^f}\big{(}{L}(\la),{N}\big{)} @>>>{\rm
Hom}_{{\mc{O}}_Y^f}\big{(}K(\la),{N}\big{)} @>>> {\rm Hom}_{{\mc{O}}_Y^f}\big{(}M,{N}\big{)}  \\
 \endCD
\end{eqnarray*}
By \lemref{ext:uhomology} $T_{\wt{K}(\la^\theta),\wt{N}}$ is an
isomorphism and by \lemref{Hom:MN} $T_{\wt{M},\wt{N}}$ is an
injection. This implies that $T_{\wt{L}(\la^\theta),\wt{N}}$ is an
isomorphism.
\end{proof}

\begin{lem}\label{lem:aux13}
Let $\la\in \mc{P}_Y$ and $\wt{N}\in\wt{\mc{O}}_Y^f$. We have
\begin{itemize}
\item[(i)] $T:{\rm
Ext}^1_{\wt{\mc{O}}_Y^f}(\wt{L}(\la^\theta),\wt{N}) \rightarrow
{\rm Ext}^1_{{\mc{O}}_Y^f}({L}(\la),N)$ is an injection.
\item[(ii)] $\ov{T}:{\rm
Ext}^1_{\wt{\mc{O}}_Y^f}(\wt{L}(\la^\theta),\wt{N})\rightarrow
{\rm Ext}^1_{\ov{{\mc{O}}}_Y^f}(\ov{L}(\la^\natural),\ov{N})$ is
an injection.
\end{itemize}
\end{lem}

\begin{proof}
Let
\begin{equation}\label{aux112}
0\longrightarrow\wt{N}\longrightarrow\wt{E}\stackrel{\wt{f}}{\longrightarrow}\wt{L}(\la^\theta)\longrightarrow 0
\end{equation}
be an exact sequence of $\DG$-modules. Suppose that
\eqnref{aux112} gives rise to a split exact sequence of
$\G$-modules
\begin{equation*}
0\longrightarrow{N}\longrightarrow{E}\stackrel{T[\wt{f}]}{\longrightarrow}{L}(\la)\longrightarrow
0.
\end{equation*}
Thus there exists $\psi\in{\rm Hom}_{\mc{O}_Y^f}(L(\la),E)$ such
that $ T[\wt{f}]\circ\psi=1_{L(\la)}$.  By \lemref{lem:aux12}
there exists $\wt{\psi}\in {\rm
Hom}_{\wt{\mc{O}}_Y^f}(\wt{L}(\la^\theta),\wt{E})$ such that
$T[\wt{\psi}]=\psi$.  Thus $T[ \wt{f}\circ\wt{\psi}]=
T[\wt{f}]\circ T[\wt{\psi}]=1_{L(\la)}$. By \lemref{lem:aux12}
we have $ \wt{f}\circ\wt{\psi}=1_{\wt{L}(\la^\theta)}$, and hence
\eqnref{aux112} is split.
\end{proof}

\begin{lem}\label{thm:aux1}
Let $\wt{M},\wt{N}\in\wt{\mc{O}}_Y^f$. We have
\begin{itemize}
\item[(i)] $T:{\rm Hom}_{\wt{\mc{O}}_Y^f}(\wt{M},\wt{N})\rightarrow {\rm Hom}_{{\mc{O}}_Y^f}({M},N)$ is an isomorphism.
\item[(ii)] $\ov{T}:{\rm Hom}_{\wt{\mc{O}}_Y^f}(\wt{M},\wt{N})\rightarrow {\rm Hom}_{\ov{{\mc{O}}}_Y^f}(\ov{M},\ov{N})$ is an isomorphism.
\end{itemize}
\end{lem}

\begin{proof}
We proceed by induction on the length of a composition series of $\wt{M}$. If $\wt{M}$ is irreducible, then it is true by \lemref{lem:aux12}.

Consider the following commutative diagram with exact top row of $\DG$-modules and exact bottom row of $\G$-modules.
\begin{eqnarray}\label{aux113}
\CD
 0  @>>> \wt{M'} @>\wt{i}>>\wt{M} @>>> \wt{L}(\la^\theta) @>>> 0 \\
 @. @VVT_{\wt{M'}}V @VVT_{\wt{M}}V @VVT_{\wt{L}(\la^\theta)}V\\
 0  @>>> M' @>i>>{M} @>>> L(\la) @>>> 0 \\
 \endCD
\end{eqnarray}
The sequence \eqnref{aux113} induces the following commutative diagram with exact rows.
\begin{eqnarray*}
\hskip -2cm\CD
0  @>>> {\rm Hom}_{\wt{\mc{O}}_Y^f}\big{(}\wt{L}(\la^\theta),\wt{N})\big{)} @>>>{\rm Hom}_{\wt{\mc{O}}_Y^f}\big{(}\wt{M},\wt{N}\big{)} @>\wt{i}^*>> \\
@. @VVT_{\wt{L}(\la^\theta),\wt{N}}V @VVT_{\wt{M},\wt{N}}V \\
0  @>>> {\rm Hom}_{{\mc{O}}_Y^f}\big{(}{L}(\la),N\big{)} @>>>{\rm Hom}_{{\mc{O}}_Y^f}\big{(}{M},N\big{)} @>i^*>> \\
\endCD
\end{eqnarray*}
\begin{eqnarray*}
\hskip 5cm\CD
@>\wt{i}^*>>{\rm Hom}_{\wt{\mc{O}}_Y^f}\big{(}\wt{M}',\wt{N}\big{)} @>>> {\rm Ext}^{1}_{\wt{\mc{O}}_Y^f}
\big{(} \wt{L}(\la^\theta),\wt{N}\big{)} \\
@. @VVT_{\wt{M}',\wt{N}}V @VVT^1_{\wt{L}(\la^\theta),\wt{N}}V \\
@>i^*>> {\rm Hom}_{{\mc{O}}_Y^f}\big{(}{M}',N\big{)} @>>> {\rm
Ext}^{1}_{{\mc{O}}_Y^f}\big{(}{L}(\la),{N}\big{)}
\endCD
\end{eqnarray*}
The map $T_{\wt{M}',\wt{N}}$ is an isomorphism by induction.  The
map $T^1_{\wt{L}(\la^\theta),\wt{N}}$ is an injection by
\lemref{lem:aux13}. Also $T_{\wt{L}(\la^\theta),\wt{N}}$ is an
isomorphism by \lemref{lem:aux12}. This implies that
$T_{\wt{M},\wt{N}}$ is an isomorphism.
\end{proof}

\begin{lem}\label{lem:aux14}
Let $\wt{M},\wt{N}\in\wt{\mc{O}}_Y^f$. We have
\begin{itemize}
\item[(i)] $T:{\rm Ext}^1_{\wt{\mc{O}}_Y^f}(\wt{M},\wt{N})\rightarrow {\rm Ext}^1_{{\mc{O}}_Y^f}({M},N)$ is an injection,
\item[(ii)] $\ov{T}:{\rm Ext}^1_{\wt{\mc{O}}_Y^f}(\wt{M},\wt{N})\rightarrow {\rm Ext}^1_{\ov{{\mc{O}}}_Y^f}(\ov{M},\ov{N})$ is an injection.
\end{itemize}
\end{lem}

\begin{proof}
The proof is virtually identical to the proof of \lemref{lem:aux13}.  Here we use
\lemref{thm:aux1} instead of \lemref{lem:aux12}.
\end{proof}

\begin{lem}\label{thm:aux2}
Let $\la\in \mc{P}_Y$ and $\wt{N}\in\wt{\mc{O}}_Y^f$. We have
\begin{itemize}
\item[(i)] $T:{\rm
Ext}^1_{\wt{\mc{O}}_Y^f}(\wt{L}(\la^\theta),\wt{N}) \rightarrow
{\rm Ext}^1_{{\mc{O}}_Y^f}({L}(\la),N)$ is an isomorphism.
\item[(ii)] $\ov{T}:{\rm
Ext}^1_{\wt{\mc{O}}_Y^f}(\wt{L}(\la^\theta),\wt{N}) \rightarrow
{\rm Ext}^1_{\ov{{\mc{O}}}_Y^f}(\ov{L}(\la^\natural),\ov{N})$ is
an isomorphism.
\end{itemize}
\end{lem}

\begin{proof}
Consider the following commutative diagram with exact rows.
\begin{eqnarray}\label{aux115}
\CD
0  @>>> \wt{M} @>>>\wt{K}(\la^\theta) @>\wt{\pi}>> \wt{L}(\la^\theta) @>>> 0 \\
@. @VVT_{\wt{M}}V @VVT_{\wt{K}(\la^\theta)}V @VVT_{\wt{L}(\la^\theta)}V\\
0  @>>> M @>>>{K(\la)} @>\pi>> L(\la) @>>> 0 \\
\endCD
\end{eqnarray}
The sequence \eqnref{aux115} induces the following commutative diagram with exact rows.
\begin{eqnarray*}
\CD
{\rm Hom}_{\wt{\mc{O}}_Y^f}\big{(}\wt{K}(\la^\theta),\wt{N}\big{)}  @>>> {\rm Hom}_{\wt{\mc{O}}_Y^f}\big{(}\wt{M},\wt{N})\big{)} @>>>{\rm
Ext}^1_{\wt{\mc{O}}_Y^f}\big{(}\wt{L}(\la^\theta),\wt{N}\big{)}@>\wt{\pi}^*>> \\
@VVT_{\wt{K}(\la^\theta),\wt{N}}V @VVT_{\wt{M},\wt{N}}V @VVT^1_{\wt{L}(\la^\theta),\wt{N}}V \\
{\rm Hom}_{{\mc{O}}_Y^f}\big{(}{K}(\la),{N}\big{)}  @>>> {\rm
Hom}_{{\mc{O}}_Y^f}\big{(}M,{N}\big{)} @>>>{\rm
Ext}^1_{{\mc{O}}_Y^f}\big{(}{L}(\la),{N}\big{)} @>\pi^*>> \\
\endCD
\end{eqnarray*}
\begin{eqnarray*}
\hskip 6cm\CD
@>\wt{\pi}^*>>{\rm
Ext}^1_{\wt{\mc{O}}_Y^f}\big{(}\wt{K}(\la^\theta),\wt{N}\big{)} @>>> {\rm Ext}_{\wt{\mc{O}}_Y^f}^1\big{(} \wt{M},\wt{N}\big{)} \\
@. @VVT^1_{\wt{K}(\la^\theta),\wt{N}}V @VVT^1_{\wt{M},\wt{N}}V \\
@>\pi^*>> {\rm Ext}^1_{{\mc{O}}_Y^f}\big{(}{K}(\la),{N}\big{)}
@>>> {\rm Ext}_{{\mc{O}}_Y^f}^1\big{(} {M},{N}\big{)}
\endCD
\end{eqnarray*}
The map $T^1_{\wt{M},\wt{N}}$ is an injection by
\lemref{lem:aux14}. The map $T^1_{\wt{K}(\la^\theta),\wt{N}}$ is
an isomorphism by \lemref{ext:uhomology}.  Also
$T_{\wt{K}(\la^\theta),\wt{N}}$ and $T_{\wt{M},\wt{N}}$ are
isomorphisms by \lemref{thm:aux1}.   Thus
$T^1_{\wt{L}(\la^\theta),\wt{N}}$ is an isomorphism.
\end{proof}

We have now all the ingredients to prove of \thmref{thm:equivalence}.
\medskip

\noindent {\em Proof of \thmref{thm:equivalence}.} \lemref{thm:aux2} implies that for
every $M\in\mc{O}_Y^f$ there exists $\wt{M}\in\wt{\mc{O}}_Y^f$ such that $T(\wt{M})=M$,
i.e.~the functor $T$ is essentially surjective.  Now \lemref{thm:aux1} says that $T$ is
full and faithful. It is well-known that an essentially surjective functor that is full
and faithful is an equivalence of categories (see e.g.~\cite{P}), proving (i). (ii) is
proved in an analogous fashion, while (iii) follows from combining (i) and (ii). \hfill
$\square$

\begin{rem}
\thmref{thm:equivalence} (iii) was stated as a conjecture in \cite[Conjecture 4.18]{CW2}.
\cite[Conjecture 4.18]{CW2} in the special case of $Y=[-m,-2]$ was already formulated in
\cite[Conjecture 6.10]{CWZ}.  A proof of \thmref{thm:equivalence} (iii) in the special
case $Y=[-m,-2]$ was announced by Brundan and Stroppel in \cite{BS}. The proof we have
presented here is different from the one announced in \cite{BS}, as it constructs
directly the functors inducing this equivalence and also does not rely on \cite{B}.
\end{rem}

\begin{rem}
The classical BGG-type resolutions for the finite-dimensional modules of \cite{Le} and
for the unitarizable modules of \cite{EW} in terms of parabolic Verma $\G$-modules
together with \thmref{thm:equivalence} imply the existence of BGG-type resolutions for
the corresponding $\DG$- and $\SG$-modules in terms of parabolic Verma $\DG$- and
$\SG$-modules, respectively. The case of $\SG$ and $Y=[-m,-2]$ was already established in
\cite{CKL}.
\end{rem}

\end{document}